\documentclass{article}
\usepackage{amsmath, amssymb}
\usepackage{amsthm}
\usepackage{multirow,bigdelim}
\usepackage{ascmac}
\usepackage{fancybox}
\usepackage{color}
\usepackage{nccmath}

\usepackage{multicol}

\usepackage[top=30truemm,bottom=30truemm,left=25truemm,right=25truemm]{geometry}
\usepackage{graphicx}
\usepackage{here}
\usepackage{amscd}
\usepackage{tikz-cd} 
\usepackage{adjustbox}

\numberwithin{equation}{section}

\usepackage{tcolorbox}
\tcbuselibrary{breakable}
\usepackage[all]{xy}

\theoremstyle{definition}
\newtheorem{thm}{Theorem}[section]

\newtheorem{lemm}[thm]{Lemma}

\newtheorem{remark}[thm]{Remark}
\newtheorem{coro}[thm]{Corollary}

\newtheorem{main}{Theorem}

\newtheorem{maincoro}{Corollary}

\begin{document}

\title{A Unified Construction of Exceptional Symmetric Spaces}

\author{Yuuki Sasaki}

\date{}

\maketitle

\begin{abstract}
We present a unified construction of all simply connected irreducible exceptional compact symmetric spaces as totally geodesic submanifolds of $E_{8}$.
Using this construction, we obtain a unified description of totally geodesic embeddings between exceptional symmetric spaces by determining all inclusion relations among them.
The resulting inclusion structure exhibits a remarkable and unexpected symmetry, suggesting a deeper geometric relationship among exceptional symmetric spaces.
\end{abstract}


\section{Introduction}

A compact symmetric space is called an exceptional symmetric space if the identity component of the isometry group is an exceptional compact Lie group.
By Cartan's classification of symmetric spaces, any simply connected irreducible exceptional compact symmetric space that is not a compact Lie group is one of the following twelve homogeneous spaces:
\[
\begin{array}{llllllll}
FI : & F_{4}/(Sp(1) \times Sp(3)/\mathbb{Z}_{2}), & 
FII : & F_{4}/Spin(9), \\
EI : & E_{6}/(Sp(4)/\mathbb{Z}_{2}), & 
EII : & E_{6}/(Sp(1) \times SU(6)/\mathbb{Z}_{2}), \\
EIII : & E_{6}/(U(1) \times Spin(10)/\mathbb{Z}_{4}), & 
EIV : & E_{6}/F_{4}, \\
EV : & E_{7}/(SU(8)/\mathbb{Z}_{2}), &
EVI : & E_{7}/(Sp(1) \cdot Spin(12)/\mathbb{Z}_{2}), \\
EVII : & E_{7}/(U(1) \times E_{6}/\mathbb{Z}_{3}), &
EVIII : & E_{8}/(Spin(16)/\mathbb{Z}_{2}), \\
EIX : & E_{8}/(Sp(1) \times E_{7}/\mathbb{Z}_{2}), & 
G :& G_{2}/SO(4).
\end{array}
\]
Throughout this paper, an \emph{exceptional symmetric space} means one of the compact irreducible symmetric spaces of types $FI, FII, EI, \cdots ,EIX$ and $G_{2}/SO(4)$.
Exceptional symmetric spaces exhibit a rich variety of remarkable geometric structures.
For example, $EIII$ and $EVII$ are Hermitian symmetric spaces, and hence K\"{a}hler symmetric spaces.
Furthermore, $G_{2}/SO(4), FI, EII, EVI$, and $EIX$ are Wolf spaces, namely quaternionic K\"{a}hler symmetric spaces.
In addition, $FII, EIII, EVI, EVIII$ admit even Clifford structures that are neither K\"{a}hler nor quaternionic K\"{a}hler \cite{Moroianu-Semmelmann}.
Thus, exceptional symmetric spaces provide important examples of Riemannian manifolds.
Nevertheless, they are usually regarded as a finite collection of individual homogeneous spaces rather than as a unified family.
Moreover, only a few geometric realizations of exceptional symmetric spaces are known.
For example, Rosenfeld introduced the Rosenfeld planes as generalizations of projective planes, and realized $FII, EIII, EVI$, and $EVIII$.
However, this construction does not realize all exceptional symmetric spaces simultaneously.
The primary aim of this paper is to give a unified construction of exceptional symmetric spaces.

Totally geodesic embeddings form an important class of embeddings between symmetric spaces.
Such embeddings between exceptional symmetric spaces were studied by Chen and Nagano \cite{Chen-Nagano1, Nagano} using the $(M^{+}, M^{-})$-method.
More recently, Kollross and Rodr\'{i}guez-V\'{a}zquez classified maximal totally geodesic submanifolds of exceptional symmetric spaces of noncompact type \cite{Kollross}.
These works treat exceptional symmetric spaces individually and do not provide a unified description of totally geodesic embeddings.
The second aim of this paper is to construct these totally geodesic embeddings between exceptional symmetric spaces in a unified way.

The construction is based on a maximal antipodal subgroup $A(E_{8})$ of order $256$ in $E_{8}$ arising in Adams' classification of maximal antipodal sets of $E_{8}$.
Maximal antipodal sets of a compact symmetric space were introduced by Chen-Nagano \cite{Chen-Nagano2}.
In a compact Lie group endowed with a bi-invariant metric, a maximal antipodal set containing the identity element is a subgroup, called a maximal antipodal subgroup.
Maximal antipodal subgroups are precisely maximal elementary abelian $2$-subgroups.
The subgroup $A(E_{8})$ is constructed using the octonions $\mathbb{O}$.
We consider the standard chain of exceptional Lie groups
\[
G_{2} \subset F_{4} \subset E_{6} \subset E_{7} \subset E_{8}.
\]
For each exceptional Lie group $G$ in the above chain and each element $p \in A(E_{8})$, we consider the conjugation orbit
\[
\mathcal{O}_{G}(p) = \{ g p g^{-1} \ ;\ g \in G\}.
\]
Our first main theorem is as follows.

\begin{main}[Theorem \ref{main-1}, \ref{main-2}]
Let $G \in \{ G_{2}, F_{4}, E_{6}, E_{7}, E_{8} \}$ and $p \in A(E_{8})$.
Every positive-dimensional orbit $\mathcal{O}_{G}(p)$ is a compact irreducible exceptional symmetric space.
Conversely, every compact irreducible exceptional symmetric space arises in this way.
Moreover, each $\mathcal{O}_{G}(p)$ is a totally geodesic submanifold of $E_{8}$.
\end{main}

Since every orbit in Theorem A is a totally geodesic submanifold of $E_{8}$, inclusions between orbits induce totally geodesic embeddings between exceptional symmetric spaces.
Our second main theorem determines all such inclusions.

\begin{main}[Theorem \ref{main-2}]
The positive-dimensional conjugation orbits $\mathcal{O}_{G}(p) \ (G \in \{ G_2, F_4, E_6, E_7, E_8 \}, p \in A(E_{8}))$ are classified as shown in the Table  \ref{all conjugation orbits}.
Moreover, all inclusion relations among these orbits, except $G_2/SO(4)$, are given by the following diagrams.
\[
\small
\xymatrix@C=8pt@R=11pt
{
 & & FII \ar[d] & & FII \ar[d] & & \\
 & & EIII \ar[d] & & EIII \ar[d] & & \\
FI \ar[r] & EII \ar[r] & EVI \ar[dr] & & EVII \ar[dl] & EIII \ar[l] & FII \ar[l] \\
 & & & EIX & & & \\
FII \ar[r] & EIV \ar[r] & EVII \ar[ur] & & EVII \ar[ul] & EIV \ar[l] & FII \ar[l] \\
 & & EIV \ar[u] & & EIV \ar[u] & & \\
 & & FII \ar[u] & & FII, \ar[u] & & \\
}
\quad
\xymatrix@C=8pt@R=11pt
{
 & & FI \ar[d] & & FI \ar[d] & & \\
 & & EII \ar[d] & & EII \ar[d] & & \\
FII \ar[r] & EIII \ar[r] & EVI \ar[dr] & & EV \ar[dl] & EII \ar[l] & FI \ar[l] \\
 & & & EVIII & & & \\
FI \ar[r] & EI \ar[r] & EV \ar[ur] & & EV \ar[ul] & EI \ar[l] & FI \ar[l] \\
 & & EI \ar[u] & & EI \ar[u] & & \\
 & & FI \ar[u] & & FI. \ar[u] & & \\
}
\]
Each orbit of type $FI$ contains exactly four distinct conjugation orbits isometric to $G_2/SO(4)$.

\begin{table}[h]
\centering
\caption{All conjugation orbits} \label{all conjugation orbits}
\begin{tabular}{lcc|lcc|lcclll} \hline
$G$ & orbit type & $\begin{matrix} \text{number of} \\ \text{distinct orbits} \end{matrix}$ & 
$G$ & orbit type & $\begin{matrix} \text{number of} \\ \text{distinct orbits} \end{matrix}$ &
$G$ & orbit type & $\begin{matrix} \text{number of} \\ \text{distinct orbits} \end{matrix}$ \\ \hline
$G_{2}$ & $G_2/SO(4)$ & 32 & $E_{6}$ & $EI$ & 4 & $E_{7}$ & $EV$ & 3 \\ 
$F_{4}$ & $FI$ & 8 & $E_{6}$ & $EII$ & 4 & $E_{7}$ & $EVI$ & 2 \\
$F_{4}$ & $FII$ & 8 & $E_{6}$ & $EIII$ & 4 & $E_{7}$ & $EVII$ & 3 \\ 
&&& $E_{6}$ & $EIV$ & 4 & $E_{8}$ & $EVIII$ & 1 \\
&&&&&& $E_{8}$ & $EIX$ & 1 \\ \hline
\end{tabular}
\end{table}

\end{main}

The classification also reveals a close relation between the conjugation orbits and the chain of division algebras
\[
\mathbb{R} \subset \mathbb{C} \subset \mathbb{H} \subset \mathbb{O}.
\]
More precisely, the orbits are characterized by the adjoint representation together with the data arising from the orthogonal complements associated with this chain.
This relation is also reflected in the number of the orbits.
Indeed, as the group becomes larger along the chain $F_{4} \subset E_{6} \subset E_{7} \subset E_{8}$, the numbers of orbits of the smaller group contained in an orbit of the larger group are 1,2, and 4.
These numbers correspond to the dimensions of the relevant orthogonal complements arising from $\mathbb{R} \subset \mathbb{C} \subset \mathbb{H} \subset \mathbb{O}$.

Combining Theorem B with nonexistence results for totally geodesic embeddings, we obtain the following criterion.

\setcounter{maincoro}{2}
\begin{maincoro}[Corollary \ref{main-3}]
Let $M$ and $N$ be distinct symmetric spaces represented in the diagram.
\[
\xymatrix@C=15pt@R=10pt
{
 & & & E_{8} & & & \\
 & & EVIII \ar[ur] & & EIX \ar[ul] & &  \\
 & EV \ar[ur] & & EVI \ar[ul] \ar[ur] & & EVII \ar[ul] & \\
 EI  \ar[ur] & & EII \ar[ul] \ar[ur] & & EIII \ar[ul] \ar[ur] & & EIV \ar[ul] \\
 & FI \ar[ul] \ar[ur] & & & & FII. \ar[ul] \ar[ur] & \\
}
\]
There exists a totally geodesic embedding $M \to N$ if and only if there is a directed path from $M$ to $N$ in the diagram.
\end{maincoro}

This diagram was essentially obtained by Chen and Nagano \cite{Chen-Nagano1} and was later displayed by Nagano \cite{Nagano}.
However, they did not provide detailed proofs of the inclusion relations appearing in the diagram.
Moreover, their approach treats exceptional symmetric spaces individually. 
In contrast, the present construction provides a unified derivation of all such inclusions.

A particularly striking feature of the diagram is its unexpected left-right symmetry, which was not discussed by Chen-Nagano or Nagano.
Interestingly, the same correspondence appears in several different geometric properties of exceptional symmetric spaces.
This indicates that the symmetry is not merely a feature of the diagram and suggests a deeper geometric relationship among exceptional symmetric spaces.
We discuss this phenomenon in Section 5.

The present paper is organized as follows.
Section 2 collects the preliminary material.
In Subsection 2.1, we recall the notion of polars and antipodal sets.
In Subsection 2.2, we briefly review root systems.
In Subsection 2.3, we recall the construction of spin groups via Clifford algebras.
Subsection 2.4 is devoted to octonions. 
We define $G_{2}$ and $Spin(8)$ using the octonions.
In Section 3, we review the construction of the exceptional group $E_{8}$ and its subgroups $E_{7}, E_{6}$, and $F_{4}$ from \cite{Adams-L}.
Section 4 contains the main results of the paper.
In Subsection 4.1, we recall the maximal antipodal subgroup of $E_{8}$ from \cite{Adams}.
We study the individual conjugation orbits in Subsections 4.2-4.6.
In Subsection 4.7, we complete the proof of the main theorem.
In Subsection 4.8, we discuss geometric features of some chains in the diagram in Corollary C.
In Section 5, we discuss the left-right symmetry of the diagram in Corollary C and its geometric implications.

\paragraph{Acknowledgments.}
This work was supported by JSPS KAKENHI Grant Number JP23K12980 and JP26K16977.
The author would like to thank Jost Eschenburg for his valuable comments and suggestions on the manuscript.


\section{Preliminaries}

We collect the notation used throughout the paper.
Let $A$ be a set and $\phi : A \to A$ be a map.  
The fixed point set of $\phi$ is denoted by $F(\phi, A)$.
Let $V$ be a vector space over $\mathbb{R}$ or $\mathbb{C}$.
For an involutive automorphism $\phi$ of $V$, denote $F^{\pm}(\phi, V) = \{ v \in V \ ;\ \phi(v) = \pm v \}$.
If $V$ is a real vector space, then its complexification is denoted by $V^{\mathbb{C}}$.
Let $G$ be a Lie group with Lie algebra $\frak{g}$.
The identity element is denoted by $e$.
Denote the center of $G$ by $C(G)$.
For $g \in G$, define $\theta_{g} : G \rightarrow G ; p \mapsto gpg^{-1}$.
Let $H$ be a subgroup of $G$.
Set $\mathcal{O}_{H}(p) = \{ \theta_{h}(p) = hph^{-1} \ ;\ h \in H \}$.
Let $W$ be a Lie subgroup of $G$ with Lie algebra $\frak{w}$.
The centralizer of $W$ in $G$ are denoted by $C(W, G)$.
Their Lie algebras are the centralizer $C(\frak{w}, \frak{g})$ of $\frak{w}$ in $\frak{g}$.
For a homomorphism $f : G \to G'$ between Lie groups, the induced homomorphism between Lie algebras will also be denoted by $f$.
Let $V$ be a vector space, and let $\rho : G \to GL(V)$ be a representation.
The induced representation of $\frak{g}$ is denoted by the same symbol $\rho$.  
For an involutive element $g \in G$, we abbreviate $F^{\pm}(\rho(g),V)$ to $F^{\pm}(g,V)$.
Finally, the diagonal matrix with diagonal components $a_{1}, \dots, a_{n}$ is denoted by $\mathrm{d}(a_{1}, \dots, a_{n})$.


\subsection{Polars and antipodal sets} \label{polars and antipodal sets}

We first recall the notion of polars introduced by Chen-Nagano \cite{Chen-Nagano1}.
Let $M$ be a connected compact symmetric space.
The symmetry at $p \in M$ is denoted by $s_{p}$.
Fix $o \in M$.
Each connected component of $F(s_{o}, M)$ is called a {\it polar} of $o$.
If a polar consists of a single point, then the polar is called a {\it pole} of $o$.
Obviously, $\{ o \}$ is a pole of $o$ and called the trivial pole.
In the present paper, the polars refer to the nontrivial polars.
Since a polar is a connected component of the fixed point set of an isometry, any polar is a totally geodesic submanifold and hence a compact symmetric space.
Let $G$ be the identity component of the isometry group of $M$ and $K_{o} = \{ g \in G \ ;\ g(o) = o \}$.
Then, $K_{o}$ acts transitively on each polar.

For example, let $H$ be a connected compact Lie group.
Then, $H$ with a bi-invariant metric is a compact symmetric space.
The geodesic symmetry at $g \in H$ is given by $s_{g}(h) = gh^{-1}g$ for $h \in H$.
Hence, $F(s_{e}, H) = \{ h \in H \ ;\ h^{2} = e \}$.
Each polar of $e$ is an orbit $\mathcal{O}_{H}(p)$, where $p$ is an involutive element.
Moreover, if $C(H)$ has a nontrivial involutive element $z$, then $z$ is a pole of $e$.
Chen-Nagano classified polars in all irreducible compact symmetric spaces \cite{Chen-Nagano1}.
Table \ref{polars} lists all polars in $G_{2}, F_{4}, E_{6}, E_{7}, E_{8}$.

\begin{table}[h]
\centering
\caption{Polars of $G_2, F_4, E_6, E_7, E_8$}
\label{polars}
\begin{tabular}{|c||c|c|c|c|c|cccccccccccccccccccccccccccc} \hline
$G$ & $G_{2}$ & $F_{4}$ & $E_{6}$ & $E_{7}$ & $E_{8}$ \\ \hline
polars & $G_{2}/SO(4)$ & $FI, FII$ & $EII, EIII$ & $EVI \times 2$, a pole & $EVIII, EIX$ \\ \hline
\end{tabular}
\end{table}

We now recall the notion of antipodal sets.
Let $A \subset M$.
If $s_{p}(q) = q$ for any $p,q \in A$, then $A$ is called an {\it antipodal set} of $M$.
It is known that any antipodal set is a finite discrete set \cite{Tanaka-Tasaki}.
If an antipodal set $A$ is maximal with respect to inclusion among antipodal sets, then $A$ is called a {\it maximal} antipodal set of $M$.
The maximum of the cardinalities of antipodal sets is called the {\it $2$-number} of $M$ and denoted by $\#_{2}M$.
An antipodal set with cardinality $\#_{2}M$ is called a {\it great} antipodal set of $M$.
A maximal antipodal set of the compact Lie group $G$ containing $e$ becomes a subgroup of $G$.
Hence, this subgroup is called a {\it maximal antipodal subgroup}.
Note that any maximal antipodal subgroup is a maximal elementary abelian $2$-subgroup. 
Conversely, any elementary abelian $2$-subgroup is an antipodal set of $G$.


\subsection{Root system}\label{Rs}

We recall some basic properties of root systems.
Let $G$ be a simply connected simple compact Lie group with Lie algebra $\frak{g}$, and let $T$ be a maximal torus of $G$ with Lie algebra $\frak{t}$.
Let $\overline{\,\cdot\,}$ be the complex conjugate of $\frak{g}^{\mathbb C}$ such that the fixed subspace is $\frak{g}$.
Let $(\, , \,)$ be the Killing form of $\frak{g}^{\mathbb C}$.
Then, $\langle X,Y\rangle = - (X,Y) \ (X, Y \in \frak{g})$ is an $\mathrm{Ad}(G)$-invariant inner product on $\frak{g}$.
The root system of $\frak{g}^{\mathbb{C}}$ with respect to $\frak{t}^{\mathbb{C}}$ is denoted by $\Sigma(\frak{g}, \frak{t})$ or $\Sigma$.
Choose a linear order on $\frak{t}$, and let $\Sigma^{+}$ denote the corresponding set of positive roots.
For any subset $B \subset \Sigma$, write $B^{+} = B \cap \Sigma^{+}$.
Let $H_\alpha \in i\frak{t}$ with $(H,H_{\alpha}) = \alpha(H)$ for any $H \in i\frak{t}$.
We denote $(\alpha, \beta)=(H_{\alpha}, H_{\beta})$ for $\alpha, \beta \in \Sigma$.
For any $\alpha \in \Sigma$ and $n \in \mathbb{Z}$, define $\Sigma_{\alpha, n}(\frak{g}, \frak{t}) = \{ \beta \in \Sigma \ ;\ 2(\alpha, \beta)/(\alpha, \alpha) = n \}$.
This set is often denoted by $\Sigma_{\alpha, n}$.
Set $A_{\alpha} = (2/(\alpha, \alpha))H_{\alpha}$ and $\tau_{\alpha} = \mathrm{exp} \ \pi(iA_{\alpha})$.
Since $\tau_{\alpha}$ is involutive, its inner automorphism $\theta_{\tau_{\alpha}}$ is involutive.
For each $\alpha\in\Sigma$, the root space is denoted by $\frak{r}_{\alpha}(\frak{g}^{\mathbb{C}}, \frak{t}^{\mathbb{C}})$ or $\frak{r}^{\mathbb{C}}_{\alpha}$.
Then, $\overline{\frak{r}_{\alpha}^{\mathbb{C}}} = \frak{r}_{-\alpha}^{\mathbb{C}}$.
We set $\frak{r}_{\alpha} = \frak{g} \cap (\frak{r}_{\alpha}^{\mathbb{C}} + \frak{r}_{-\alpha}^{\mathbb{C}})$.
Accordingly,
\[
\frak{g} = \frak{t} \oplus \bigoplus_{\alpha \in \Sigma^{+}}\frak{r}_{\alpha}
\]
is an orthogonal direct sum decomposition.
For each $\alpha\in\Sigma$, set
\[
\frak{su}^{\alpha}(2) = \mathbb{R}(iA_{\alpha}) \oplus \frak{r}_{\alpha}.
\]
Then, $\frak{su}^{\alpha}(2)$ is a subalgebra of $\frak{g}$ and isomorphic to $\frak{su}(2)$.
Let $\frak{t}^{\alpha}$ be the orthogonal complement of $\mathbb{R}(iA_{\alpha})$ in $\frak{t}$.
We have the orthogonal direct sum decomposition
\[
C(\frak{su}^{\alpha}(2), \frak{g}) = \frak{t}^{\alpha} \oplus \bigoplus_{\gamma \in \Sigma_{\alpha, 0}^{+}}\frak{r}_{\gamma}.
\]
If $\alpha, \beta \in \Sigma$ satisfy $(\alpha, \alpha) = (\beta, \beta)$ and $2(\alpha, \beta)/(\alpha, \alpha) = \pm1$, then $\{ \pm \alpha, \pm \beta, \pm (\alpha \mp \beta) \}$ is a root system of type $A_{2}$.
Set
\[
\frak{su}^{\alpha, \beta}(3) = \mathbb{R}(iA_{\alpha}) \oplus \mathbb{R}(iA_{\beta}) \oplus \bigoplus_{\gamma \in \{ \alpha, \beta, \alpha \mp \beta \} }\frak{r}_{\gamma}.
\]
Then, $\frak{su}^{\alpha, \beta}(3)$ is a subalgebra of $\frak{g}$ and isomorphic to $\frak{su}(3)$.
Let $\frak{t}^{\alpha, \beta}$ be the orthogonal complement of $\mathbb{R}(iA_{\alpha}) \oplus \mathbb{R}(iA_{\beta})$ in $\frak{t}$.
We have the orthogonal direct sum decomposition 
\[
C(\frak{su}^{\alpha, \beta}(3), \frak{g}) = \frak{t}^{\alpha, \beta} \oplus \bigoplus_{\gamma \in (\Sigma_{\alpha, 0} \cap \Sigma_{\beta, 0})^{+}}\frak{r}_{\gamma}.
\]


\subsection{Spin group}\label{Sg}

In this subsection, we recall the Clifford algebra, the spin group, and the spin representation.
Let $(\ ,\ )$ denote the standard inner product on $\mathbb{R}^{n}$ and let $T(\mathbb{R}^{n})$ be the tensor algebra over $\mathbb{R}^{n}$.
Let $J$ be the two-sided ideal of $T(\mathbb{R}^{n})$ generated by $v \otimes v + (v,v)1$ for any $v \in \mathbb{R}^{n}$.
The {\it Clifford algebra} over $\mathbb{R}^{n}$ is 
\[
Cl(\mathbb{R}^{n}) = T(\mathbb{R}^{n})/J.
\]
The product of $u, v \in Cl(\mathbb{R}^{n})$ is denoted by $uv$.
Let $e_{0}, \cdots, e_{n-1}$ be the standard orthonormal basis of $\mathbb{R}^{n}$.
The defining relation is $e_{i}e_{j} + e_{j}e_{i} = -2\delta_{ij}$ for $0 \leq i, j \leq n-1$.
Let $\pi : T(\mathbb{R}^{n}) \rightarrow Cl(\mathbb{R}^{n})$ be the natural projection.
For $u \in \mathbb{R}^{n} \subset Cl(\mathbb{R}^{n})$, the inverse of $u$ in the Clifford algebra is denoted by $u^{-1}$.
If $u \not= 0$, then $u^{-1} = (-1/||u||^{2})u$.
In particular, if $||u|| = 1$, the inverse is $u^{-1} = -u$.
Let $v \in \mathbb{R}^{n}$ with $||v|| = 1$.
For $x \in \mathbb{R}^{n}$, since $vxv^{-1} = -x + 2(v,x)v$, the map $\mathbb{R}^{n} \ni x \mapsto vxv^{-1} \in \mathbb{R}^{n}$ is the composition of the reflection in the hyperplane orthogonal to $v$ with multiplied by $-1$.
Set 
\[
Pin(n) = \{ v_{1}v_{2} \cdots v_{p} \ ;\ v_{i} \in \mathbb{R}^{n}, ||v_{i}|| = 1 \} \subset Cl(\mathbb{R}^{n}).
\]
This is a compact Lie group and called the {\it pin group}.
It follows that
\[
\pi : Pin(n) \rightarrow O(n), \ \pi(g)(x)=gxg^{-1} \ (g \in Pin(n),\ x \in \mathbb{R}^{n})
\]
is a double covering homomorphism.
The map $Pin(n) \ni g \mapsto -g \in Pin(n)$ is the corresponding deck transformation.
Define
\[
T(\mathbb{R}^{n})^{0} = \bigoplus_{k\ge0} \otimes^{2k} \mathbb{R}^{n},
\qquad
T(\mathbb{R}^{n})^{1} = \bigoplus_{k\ge0} \otimes^{2k+1} \mathbb{R}^{n}.
\]
Set
\[
Spin(n) = Pin(n) \cap Cl(\mathbb{R}^{n})^{0} = \{ v_{1} \cdots v_{2m} \in Pin(n) \ ;\ v_{i} \in \mathbb{R}^{n}, ||v_{i}|| = 1\}.
\]
Then, $Spin(n)$ is a compact Lie group and called the {\it spinor group}.
The group $Pin(n)$ has two connected components, and $Spin(n)$ is the identity component.
The restriction of $\pi$ yields a double covering $\pi : Spin(n) \rightarrow SO(n)$.
The center $C(Spin(n))$ of $Spin(n)$ is 
\[
C(Spin(n)) = 
\begin{cases}
\{ \pm 1, \pm e_{1}e_{2} \cdots e_{n} \} \cong \mathbb{Z}_{2} \times \mathbb{Z}_{2} & (n = 4m), \\
\{ \pm 1, \pm e_{1}e_{2} \cdots e_{n} \} \cong \mathbb{Z}_{4} & (n = 4m+2), \\
\{ \pm 1 \} \cong \mathbb{Z}_{2} & (n = 2m+1). \\
\end{cases}
\]

Throughout this subsection, we assume that $n = 2m$.
The Lie algebra of $Spin(n)$ is 
\[
\frak{spin}(n) = \Big\{ \sum_{0 \leq i < j \leq n-1}c_{ij}e_{i}e_{j} \in Cl(\mathbb{R}^{n})^{0} \ ;\ c_{ij} \in \mathbb{R} \Big\}.
\]
The Lie bracket $[\ ,\ ]$ is given by $[X,Y] = XY - YX$ for $X,Y \in \frak{spin}(n)$, where $XY$ is the product of $X$ and $Y$ in $Cl(\mathbb{R}^{n})$.
Let $E_{ij} \in \frak{so}(n) \ (i < j)$ be defined by
\[
E_{ij}(e_{k}) = 
\begin{cases}
e_{j} & ( k = i ) , \\
-e_{i} & ( k = j ), \\
0 & (\text{otherwise}).
\end{cases}
\]
We have $\pi(e_{i}e_{j}) = 2E_{ij}$.
The adjoint representation is given by $\mathrm{Ad}(g)X = gXg^{-1}$ for any $g \in Spin(n)$ and $X \in \frak{spin}(n)$.
For any $\theta_{1}, \cdots, \theta_{m} \in \mathbb{R}$, set
\[
\begin{split}
& t_{n}(\theta_{1}, \cdots, \theta_{m}) = \sum_{i=1}^{m}\frac{\theta_{i}}{2}e_{2i-2}e_{2i-1} \in \frak{spin}(n), \\
& t_{n}'(\theta_{1}, \cdots, \theta_{m}) = \sum_{i=1}^{m}\theta_{i}E_{2i-2, 2i-1} \in \frak{so}(n). \\
\end{split}
\]
It follows that the subspaces 
\[
\begin{split}
& \frak{t}_{n} = \{ t_{n}(\theta_{1}, \cdots, \theta_{m}) \ ;\ \theta_{i} \in \mathbb{R}\ (1 \leq i \leq m)\}, \\
& \frak{t}_{n}' = \{ t_{n}'(\theta_{1}, \cdots, \theta_{m}) \ ;\ \theta_{i} \in \mathbb{R}\ (1 \leq i \leq m)\}
\end{split}
\] 
are maximal abelian subalgebras of $\frak{spin}(n)$ and $\frak{so}(n)$.
Set
\[
\begin{split}
T_{n}(\theta_{1}, \cdots, \theta_{m}) 
& = \mathrm{exp} t_{n}(\theta_{1}, \cdots, \theta_{m}) = \prod_{i=1}^{m} \left (\cos \frac{\theta_{i}}{2} + \sin \frac{\theta_{i}}{2} e_{2i-2}e_{2i-1} \right) \in Spin(n), \\
T'_{n}(\theta_{1}, \cdots, \theta_{m}) 
&= \mathrm{exp} t_{n}'(\theta_{1}, \cdots, \theta_{m}) \\
&= d( R(\theta_{1}), \cdots, R(\theta_{m})) \in SO(n)
 \quad \left( R(\theta) = \begin{pmatrix} \cos \theta & -\sin \theta \\ \sin \theta & \cos \theta \end{pmatrix} \right).
\end{split}
\]
The subgroups $T_{n} = \{ T_{n}(\theta_{1}, \cdots, \theta_{m}) \ ;\ \theta_{i} \in \mathbb{R} \}$ and $T'_{n} = \{ T'_{n}(\theta_{1}, \cdots, \theta_{m}) \ ;\ \theta_{i} \in \mathbb{R} \}$ are maximal tori of $Spin(n)$ and $SO(n)$.
Since $\pi(t_{n}(\theta_{1}, \cdots, \theta_{m})) = t_{n}'(\theta_{1}, \cdots, \theta_{m})$, we have $\pi(T_{n}(\theta_{1}, \cdots, \theta_{m})) = T'_{n}(\theta_{1}, \cdots, \theta_{m})$.
Let $x_{i}$ be the 1-form of $\frak{t}_{n}^{\mathbb{C}}$ such that 
\[
x_{i}( t_{n}(\theta_{1}, \cdots, \theta_{m}) ) = i\theta_{i}.
\]
The root system of $\frak{spin}(n)^{\mathbb{C}}$ with respect to $\frak{t}^{\mathbb{C}}$ is 
\[
\Sigma(\frak{spin}(n), \frak{t}_{n}) = \{ \pm x_{i} \pm x_{j} \ ;\ 1 \leq i < j \leq m \}.
\]
The corresponding root spaces are
\[
\begin{split}
\frak{r}_{\pm(x_{i}-x_{j})}^{\mathbb{C}} &= \mathbb{C}( (e_{2i-1}e_{2j-1} + e_{2i}e_{2j}) \mp i(e_{2i}e_{2j-1} - e_{2i-1}e_{2j}) ), \\
\frak{r}_{\pm(x_{i}+x_{j})}^{\mathbb{C}} &= \mathbb{C}( (e_{2i-1}e_{2j-1} - e_{2i}e_{2j}) \mp i(e_{2i}e_{2j-1} + e_{2i-1}e_{2j}) ).
\end{split}
\]
The Killing form is given by $(X, Y) = (n-2)\mathrm{tr}(\pi(X)\pi(Y))$ for any $X,Y \in \frak{spin}(n)$.
Hence, 
\[
iA_{x_{i} \pm x_{j}} = -(1/2)(e_{2i-2}e_{2i-1} \pm e_{2j-2}e_{2j-1}).
\]
We take a linear order of $i\frak{t}$ such that $\Sigma^{+} = \{ x_{i} \pm x_{j} \ ;\ 1 \leq i < j \leq m \}$.


Let $p,q \in \mathbb{N}$ with $p+q=n$.
The Clifford algebras $Cl(\mathrm{span}_{\mathbb{R}}\{ e_{0}, \cdots, e_{p-1} \})$ and $Cl(\mathrm{span}_{\mathbb{R}}\{ e_{p}, \cdots, e_{p+q-1} \})$ give rise to the groups $Spin(p)$ and $Spin(q)$.
Multiplication in $Cl(\mathbb{R}^{n})$ induces a homomorphism
\[
i_{p,q}\ :\ Spin(p) \times Spin(q) \rightarrow Spin(n) \ ;\ (g,h) \mapsto gh.
\]
Its kernel is $\mathrm{Ker}\ i_{p,q} = \{ (1,1), (-1,-1) \}$. 
The image is isomorphic to $(Spin(p) \times Spin(q))/\mathbb{Z}_{2}$ and denoted by $Spin(p) \cdot Spin(q)$.

Let $\alpha \in \Sigma^{+}$.
The connected subgroup with Lie algebra $\frak{su}^{\alpha}(2)$ is denoted by $SU^{\alpha}(2)$.
It is isomorphic to $SU(2)$.
For $\alpha = x_{i} \pm x_{j} \ (1 \leq i < j \leq m)$, denote $\bar{\alpha} = x_{i} \mp x_{j}$.
We have the orthogonal direct sum decomposition
\[
C(\frak{su}^{\alpha}(2), \frak{spin}(n)) = \frak{t}^{\alpha} \oplus \bigoplus_{\beta \in \Sigma^{+}_{\alpha,0}} \frak{r}_{\beta} \cong \frak{su}(2) \oplus \frak{spin}(n-4).
\]
The irreducible component isomorphic to $\frak{su}(2)$ is $\frak{su}^{\bar{\alpha}}(2)$.
The other component is denoted by $\frak{spin}^{\alpha}(n-4)$.

Let $\beta \in \Sigma^{+}$ satisfy $2(\alpha, \beta)/(\alpha, \alpha) = \pm 1$.
In this case, $\alpha \mp \beta \in \Sigma^{+}$.
The connected subgroup with Lie algebra $\frak{su}^{\alpha, \beta}(3)$ is isomorphic to $SU(3)$ and denoted by $SU^{\alpha, \beta}(3)$.
We have the orthogonal direct sum decomposition
\[
C(\frak{su}^{\alpha,\beta}(3), \frak{spin}(n)) = \frak{t}^{\alpha, \beta} \oplus \bigoplus_{\gamma \in (\Sigma_{\alpha, 0} \cap \Sigma_{\beta, 0})^{+}}\frak{r}_{\gamma} \cong \mathbb{R} \oplus \frak{spin}(n-6).
\]
The irreducible component isomorphic to $\frak{spin}(n-6)$ is denoted by $\frak{spin}^{\alpha, \beta}(n-6)$.


Let $(\Delta_{2m}^{\pm})^{\mathbb{C}}$ be the spin representations of $Spin(2m)$, and the representation spaces are denoted by $(V_{2m}^{\pm})^{\mathbb{C}}$.
The spin representations are irreducible.
The spin representations $(\Delta_{m}^{\pm})^{\mathbb{C}}$ admit real representations $\Delta_{m}^{\pm}$ if and only if $m \equiv 0 \mod 8$, and the real representation spaces are denoted by $V_{m}^{\pm}$.
The set of weights of $(\Delta_{2m}^{\pm})^{\mathbb{C}}$ with respect to $\frak{t}_{2m}^{\mathbb{C}}$ is 
\[
\Pi_{2m}^{\pm} = \left\{ \frac{1}{2} \Big( \epsilon_{1}x_{1} + \cdots + \epsilon_{m}x_{m} \Big) \ ;\ \epsilon_{i} = \pm1 \ (i = 1, \cdots, m), \epsilon_{1} \cdots \epsilon_{m} = \pm 1 \right\}.
\]
For $s, r \in \mathbb{N}$ with $s + r = m$, we have
\[
\begin{array}{llll}
(\Delta_{2r+ 2s}^{+})^{\mathbb{C}} = \Big( (\Delta_{2r}^{+})^{\mathbb{C}} \otimes (\Delta_{2s}^{+})^{\mathbb{C}} \Big) \oplus \Big( (\Delta^{-}_{2r})^{\mathbb{C}} \otimes (\Delta^{-}_{2s})^{\mathbb{C}} \Big)
\end{array}
\]
as representations of $Spin(2r) \times Spin(2s)$.


\subsection{Octonions, $G_{2}$, and $Spin(8)$}\label{O}

We briefly recall the octonions, $G_{2}$, and $Spin(8)$ from \cite{Yokota}.
Let $\mathbb{O}= \bigoplus_{i=0}^{7}\mathbb{R}e_{i}$ be the octonions, where $e_{0}, \cdots, e_{7}$ form a basis.
The multiplication on $\mathbb O$, denoted by $\cdot$, is defined as follows.
The element $e_{0}$ serves as the unit element and will be written simply as $1$.
For each $1 \leq i \not= j \leq 7$, assume $e_{i} \cdot e_{i} = -1$ and $e_{i} \cdot e_{j} = -e_{j} \cdot e_{i}$.
The multiplication is distributive.
In Figure $1$, the multiplication among $e_{1},e_{2},e_{3}$ is defined as $e_{1} \cdot e_{2} = e_{3}, \ e_{2} \cdot e_{3}=e_{1}, \ e_{3} \cdot e_{1}=e_{2}$.
The multiplication on each of the remaining lines and on the circle is defined in the same way.
We equip $\mathbb{O}$ with the standard inner product such that $( \sum_{i=0}^{7}x_{i}e_{i}, \sum_{i=0}^{7}y_{i}e_{i}) = \sum_{i=0}^{7}x_{i}y_{i} \ (x_{i}, y_{i} \in \mathbb{R})$.
Using this inner product, we identify $SO(\mathbb{O})$ and $\frak{so}(\mathbb{O})$ with $SO(8)$ and $\frak{so}(8)$, respectively.
A linear automorphism $f:\mathbb{O} \rightarrow \mathbb{O}$ is called an automorphism of $\mathbb{O}$ if $f(x \cdot y) = f(x) \cdot f(y) \ (x,y \in \mathbb{O})$.
The automorphism group of $\mathbb{O}$ is the compact exceptional Lie group $G_{2}$.
The group $G_{2}$ is connected and simply connected.
Furthermore, $G_{2} \subset SO(8)$.
The Lie algebra of $G_{2}$ is given by $\frak{g}_{2} = \{ X \in \frak{so}(8) \ ;\ X(u) \cdot v + u \cdot X(v) = X(u \cdot v) \ (u,v \in \mathbb{O}) \}$.

\begin{figure}[h]
\centering
\includegraphics[width=40mm]{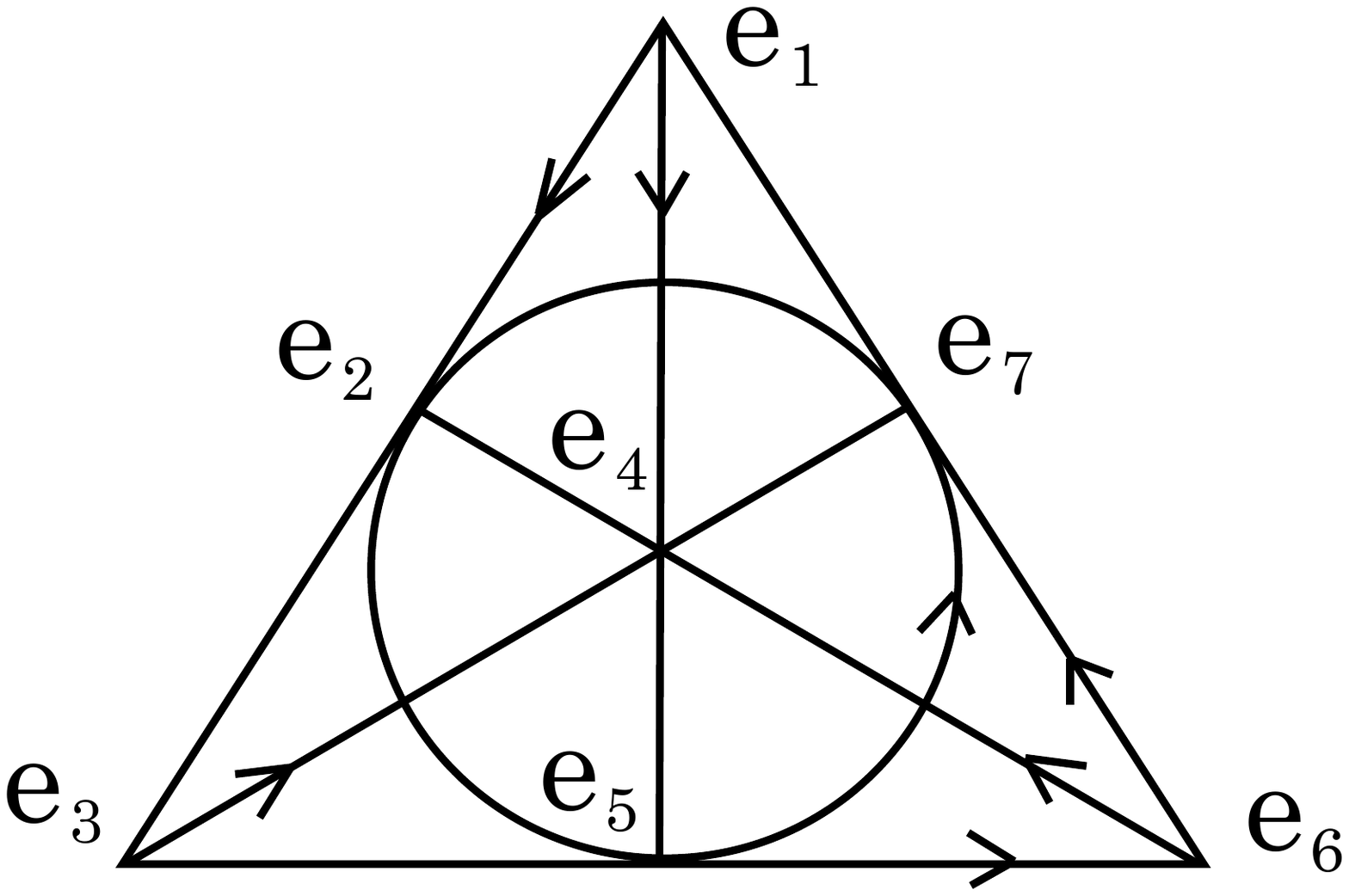}
\caption{Octonions}\label{octonion}
\end{figure}


For any $g_{1} \in SO(8)$, there exist $g_{2}, g_{3} \in SO(8)$ such that 
\[
(g_{1}u) \cdot (g_{2}v) = g_{3}(u \cdot v)\] 
for any $u,v \in \mathbb{O}$.
Such a pair $(g_{2}, g_{3})$ is unique up to replacing it by $(-g_{2}, -g_{3})$.
Set a Lie subgroup of $SO(8)^{3}$ by 
\[
D_{4} = \{ (x_{1}, x_{2}, x_{3}) \in SO(8)^{3} \ ;\ (x_{1}u) \cdot (x_{2}v) = x_{3}(u \cdot v) \ (u,v \in \mathbb{O}) \}.
\]
Similarly, for any $X_{1} \in \frak{so}(8)$, there exist unique $X_{2}, X_{3} \in \frak{so}(8)$ such that 
\[
X_{1}(u) \cdot v + u \cdot (X_{2}v) = X_{3}(u \cdot v)\]
for any $u, v \in \mathbb{O}$.
The Lie algebra of $D_{4}$ is
\[
\frak{d}_{4} = \{ (X_{1}, X_{2}, X_{3}) \in \frak{so}(8) \oplus \frak{so}(8) \oplus \frak{so}(8) \ ;\ (X_{1}u) \cdot v + u \cdot (X_{2}v) = X_{3}(u \cdot v) \ (u,v \in \mathbb{O}) \}.
\]
The projection $\pi_{\mathbb{O}}:D_{4} \rightarrow SO(8) \ ;\ (x_{1}, x_{2}, x_{3}) \mapsto x_{1}$ is a double covering homomorphism.
Consequently, $D_{4}$ is isomorphic to $Spin(8)$, and $\frak{d}_{4} \ni (X_{1}, X_{2}, X_{3}) \mapsto X_{1} \in \frak{so}(8)$ is an isomorphism.
The subgroup $\{ (x,x,x) \in D_{4} \ ;\ x \in SO(8) \}$ of $D_{4}$ is naturally identified with $G_{2}$.
For each $g \in G_{2}$, define the elements of $D_{4}$ by
\[
g^{0,0} = (g,g,g), \qquad 
g^{1,0} = (g,-g,-g), \qquad 
g^{0,1} = (-g,-g,g), \qquad 
g^{1,1} = (-g,g,-g).
\]
We have $C(D_{4}) = \{ \mathrm{id}_{\mathbb{O}}^{a,b} \ ;\ a,b = 0,1 \} \cong \mathbb{Z}_{2} \times \mathbb{Z}_{2}$.

Let $a$ be an imaginary octonion.
Set $L_{a}, R_{a}, T_{a} \in \frak{so}(8)$ by $L_{a}(x) = a \cdot x, \ R_{a}(x) = x \cdot a$, and $T_{a} = L_{a} + R_{a}$ for any $x \in \mathbb{O}$.
Since $(T_{a}u)v - u(L_{a}v) = L_{a}(uv)$ for any $u,v \in \mathbb{O}$, we have $(T_{a}, -L_{a}, L_{a}) \in \frak{d}_{4}$.
Since $T_{e_{i}} = 2E_{0i}$ and $[T_{e_{i}}, T_{e_{j}}] = 4E_{ij}$,
\[
\widetilde{E}_{0i} = \Big( E_{0i}, -\frac{1}{2}L_{e_{i}}, \frac{1}{2}L_{e_{i}} \Big), \quad
\widetilde{E}_{ij} = \Big( E_{ij}, \frac{1}{4}[L_{e_{i}}, L_{e_{j}}], \frac{1}{4}[L_{e_{i}}, L_{e_{j}}] \Big) \quad (1 \leq i < j \leq 7)
\]
form a basis of $\frak{d}_{4}$.
For any $\theta_{1}, \cdots, \theta_{4} \in \mathbb{R}$, set
\[
\widetilde{t}_{8}(\theta_{1}, \cdots, \theta_{4}) = \sum_{i=1}^{4} \theta_{i} \widetilde{E}_{2i-2, 2i-1}.
\]
The subspace $\widetilde{\frak{t}}_{8} = \{ \widetilde{t}_{8}(\theta_{1}, \cdots, \theta_{4}) \ ;\ \theta_{i} \in \mathbb{R} \ (1 \leq i \leq 4) \}$ is a maximal abelian subalgebra of $\frak{d}_{4}$.
Note that
\[
\begin{array}{lll}
\ L_{e_{1}} = E_{01} + E_{23} + E_{45} + E_{67}, & [L_{e_{2}}, L_{e_{3}}] = 2(E_{01} + E_{23} - E_{45} - E_{67}), \\ 
\ [L_{e_{4}}, L_{e_{5}}] = 2(E_{01} - E_{23} + E_{45} - E_{67}), & [L_{e_{6}}, L_{e_{7}}] = 2(E_{01} - E_{23} - E_{45} + E_{67}). \\
\end{array}
\]
Set
\[
\small
\begin{split}
& \widetilde{T}_{8}(\theta_{1}, \cdots, \theta_{4}) = \mathrm{exp}( \widetilde{t}_{8}(\theta_{1}, \cdots, \theta_{4}) ) \\
&= \Big( T_{8}'(\theta_{1}, \theta_{2}, \theta_{3}, \theta_{4}),\\
& \quad\quad T'_{8} \left( \frac{-\theta_{1} + \theta_{2} + \theta_{3} + \theta_{4}}{2}, \frac{-\theta_{1} + \theta_{2} - \theta_{3} - \theta_{4}}{2}, \frac{-\theta_{1} - \theta_{2} + \theta_{3} - \theta_{4}}{2}, \frac{-\theta_{1} - \theta_{2} - \theta_{3} + \theta_{4}}{2} \right), \\
& \quad\quad T'_{8} \left( \frac{+\theta_{1} + \theta_{2} + \theta_{3} + \theta_{4}}{2}, \frac{+\theta_{1} + \theta_{2} - \theta_{3} - \theta_{4}}{2}, \frac{+\theta_{1} - \theta_{2} + \theta_{3} - \theta_{4}}{2}, \frac{+\theta_{1} - \theta_{2} - \theta_{3} + \theta_{4}}{2} \right)
\Big). \\
\end{split}
\]
The subgroup $\widetilde{T}_{8} = \{ \widetilde{T}_{8}(\theta_{1}, \cdots, \theta_{4}) \ ;\ \theta_{i} \in \mathbb{R} \ (1 \leq i \leq 4) \}$ is a maximal torus of $D_{4}$.
We have $\pi_{\mathbb{O}}(\widetilde{T}_{8}(\theta_{1}, \cdots, \theta_{4})) = T'_{8}(\theta_{1}, \cdots, \theta_{4})$.

Let $\lambda : Spin(8) \rightarrow D_{4}$ be the isomorphism such that the derivation $\lambda : \frak{spin}(8) \rightarrow \frak{d}_{4}$ satisfies $\lambda = \pi_{\mathbb{O}}^{-1} \circ \pi$, where $Spin(8)$ is defined by $Cl(\mathrm{span}_{\mathbb{R}}\{ e_{0}, \cdots, e_{7} \})$.
We identify $Spin(8)$ and $\frak{spin}(8)$ with $D_{4}$ and $\frak{d}_{4}$ under $\lambda$. 
The center $C(Spin(8))$ corresponds to the center $C(D_{4})$ as 
\[
1 = \mathrm{id}_{\mathbb{O}}^{0,0},  \quad
-1 = \mathrm{id}_{\mathbb{O}}^{1,0}, \quad 
e_{0}e_{1} \cdots e_{7} = \mathrm{id}_{\mathbb{O}}^{0,1}, \quad 
-e_{0}e_{1} \cdots e_{7} = \mathrm{id}_{\mathbb{O}}^{1,1}.
\]
Under this identification, $\widetilde{T}_{8}(\theta_{1}, \cdots, \theta_{4})$ and $\widetilde{t}_{8}(\theta_{1}, \cdots, \theta_{4})$ correspond to $T_{8}(\theta_{1}, \cdots, \theta_{4})$ and $t_{8}(\theta_{1}, \cdots, \theta_{4})$.


Let $\tilde{\Delta}_{8}^{\pm}$ denote the real representations of $Spin(8)$ on $\mathbb{O}$ defined by
\[
\tilde{\Delta}_{8}^{+}(g_{1}, g_{2}, g_{3}) = g_{3}(u), \quad
\tilde{\Delta}_{8}^{-}(g_{1}, g_{2}, g_{3}) = g_{2}(u) \quad ((g_{1}, g_{2}, g_{3}) \in Spin(8), u \in \mathbb{O}).
\]
We write their complexifications as $(\tilde{\Delta}_{8}^{\pm})^{\mathbb{C}}$.
Let $\tilde{\Pi}^{\pm}_{8}$ denote the set of weights of $(\tilde{\Delta}^{\pm}_{8})^{\mathbb{C}}$ with respect to $\frak{t}_{8}^{\mathbb{C}}$.
The sets of weights are
\[
\begin{split}
\tilde{\Pi}_{8}^{\pm} &= \left\{ \frac{1}{2} \sum_{i=1}^{4} \epsilon_{i}x_{i} \ ;\ \epsilon_{i} = 1 \text{\ or\ }-1, \ \epsilon_{1} \epsilon_{2} \epsilon_{3}  \epsilon_{4} = \pm 1 \right\}. \\
\end{split}
\]
For $\gamma \in \tilde{\Pi}_{8}^{\pm}$, write $(V_{8}^{\pm})^{\mathbb{C}}_{\gamma}$ for the corresponding weight space.
Then,
\[
\begin{array}{lllllll}
(V_{8}^{+})_{\pm\frac{1}{2}(x_{1} + x_{2} + x_{3} + x_{4})}^{\mathbb{C}} = \mathbb{C}(e_{0} \mp ie_{1}), &
(V_{8}^{-})_{\pm\frac{1}{2}(-x_{1} + x_{2} + x_{3} + x_{4})}^{\mathbb{C}} = \mathbb{C}(e_{0} \mp ie_{1}), \\
(V_{8}^{+})_{\pm\frac{1}{2}(x_{1} + x_{2} - x_{3} - x_{4})}^{\mathbb{C}} = \mathbb{C}(e_{2} \mp ie_{3}), &
(V_{8}^{-})_{\pm\frac{1}{2}(-x_{1} + x_{2} - x_{3} - x_{4})}^{\mathbb{C}} = \mathbb{C}(e_{2} \mp ie_{3}), \\
(V_{8}^{+})_{\pm\frac{1}{2}(x_{1} - x_{2} + x_{3} - x_{4})}^{\mathbb{C}} = \mathbb{C}(e_{4} \mp ie_{5}), & 
(V_{8}^{-})_{\pm\frac{1}{2}(-x_{1} - x_{2} + x_{3} - x_{4})}^{\mathbb{C}} = \mathbb{C}(e_{4} \mp ie_{5}), \\
(V_{8}^{+})_{\pm\frac{1}{2}(x_{1} - x_{2} - x_{3} + x_{4})}^{\mathbb{C}} = \mathbb{C}(e_{6} \mp ie_{7}), &
(V_{8}^{-})_{\pm\frac{1}{2}(-x_{1} - x_{2} - x_{3} + x_{4})}^{\mathbb{C}} = \mathbb{C}(e_{6} \mp ie_{7}). \\
\end{array}
\]
It follows that $\tilde{\Delta}_{8}^{\pm}$ and $(\tilde{\Delta}_{8}^{\pm})^{\mathbb{C}}$ are isomorphic to the spin representations $\Delta_{8}^{\pm}$ and $(\Delta_{8}^{\pm})^{\mathbb{C}}$.
Thus, $V_{8}^{\pm}$ and $(V_{8}^{\pm})^{\mathbb{C}}$ are identified with $\mathbb{O}$ and $\mathbb{O}^{\mathbb{C}}$.
When we regard $\mathbb{O}$ and $\mathbb{O}^{\mathbb{C}}$ as the representation spaces of $\Delta_{8}^{\pm}$ and $(\Delta_{8}^{\pm})^{\mathbb{C}}$, we write $\mathbb{O}$ and $\mathbb{O}^{\mathbb{C}}$ by $\mathbb{O}^{\pm}$ and $(\mathbb{O}^{\pm})^{\mathbb{C}}$.


We introduce some elements of $G_{2}$ and $Spin(8)$ that we will use later.
Define the following elements of $G_{2}$:
\[
\begin{array}{llll}
\gamma_{0} = \mathrm{id}_{\mathbb{O}}, &
\gamma_{1} = \mathrm{id}_{\mathrm{span}_{\mathbb{R}}\{ e_{0}, e_{1}, e_{2}, e_{3} \}} - \mathrm{id}_{\mathrm{span}_{\mathbb{R}}\{ e_{4}, e_{5}, e_{6}, e_{7} \}}, \\
\gamma_{2} = \mathrm{id}_{\mathrm{span}_{\mathbb{R}}\{ e_{0}, e_{1}, e_{4}, e_{5} \}} - \mathrm{id}_{\mathrm{span}_{\mathbb{R}}\{ e_{2}, e_{3}, e_{6}, e_{7} \}}, &
\gamma_{3} = \mathrm{id}_{\mathrm{span}_{\mathbb{R}}\{ e_{0}, e_{1}, e_{6}, e_{7} \}} - \mathrm{id}_{\mathrm{span}_{\mathbb{R}}\{ e_{2}, e_{3}, e_{4}, e_{5} \}}, \\
\gamma_{4} = \mathrm{id}_{\mathrm{span}_{\mathbb{R}}\{ e_{0}, e_{2}, e_{4}, e_{6} \}} - \mathrm{id}_{\mathrm{span}_{\mathbb{R}}\{ e_{1}, e_{3}, e_{5}, e_{7} \}}, &
\gamma_{5} = \mathrm{id}_{\mathrm{span}_{\mathbb{R}}\{ e_{0}, e_{2}, e_{5}, e_{7} \}} - \mathrm{id}_{\mathrm{span}_{\mathbb{R}}\{ e_{1}, e_{3}, e_{4}, e_{6} \}}, \\
\gamma_{6} = \mathrm{id}_{\mathrm{span}_{\mathbb{R}}\{ e_{0}, e_{3}, e_{4}, e_{7} \}} - \mathrm{id}_{\mathrm{span}_{\mathbb{R}}\{ e_{1}, e_{2}, e_{5}, e_{6} \}}, &
\gamma_{7} = \mathrm{id}_{\mathrm{span}_{\mathbb{R}}\{ e_{0}, e_{3}, e_{5}, e_{6} \}} - \mathrm{id}_{\mathrm{span}_{\mathbb{R}}\{ e_{1}, e_{2}, e_{4}, e_{7} \}}. \\
\end{array}
\]
The elements $\gamma_{1}, \cdots, \gamma_{7}$ correspond to the lines and the circle in Figure \ref{octonion}: the $+1$-eigenspace og $\gamma_{i} \ (1 \leq i \leq 7)$ is spanned by $e_{0}$ and three basis elements $e_{j}$ lying on the corresponding line or circle.
Thus, the definition of these elements directly reflects the multiplication structure of the octonions.
Under the above identification $D_{4} \cong Spin(8)$,
\[
\begin{array}{llllllllllllll}\vspace{2mm}
\gamma_{0}^{0,0} = 1, & \gamma_{0}^{1,0} = -1, & \gamma_{0}^{0,1} = e_{0} \cdots e_{7}, & \gamma_{0}^{1,1} = -e_{0} \cdots e_{7}, \\\vspace{2mm}
\gamma_{1}^{0,0} = -e_{4}e_{5}e_{6}e_{7}, & \gamma_{1}^{1,0} = e_{4}e_{5}e_{6}e_{7}, & \gamma_{1}^{0,1} = -e_{0}e_{1}e_{2}e_{3}, & \gamma_{1}^{1,1} = e_{0}e_{1}e_{2}e_{3}, \\\vspace{2mm}
\gamma_{2}^{0,0} = -e_{2}e_{3}e_{6}e_{7}, & \gamma_{2}^{1,0} = e_{2}e_{3}e_{6}e_{7}, & \gamma_{2}^{0,1} = -e_{0}e_{1}e_{4}e_{5}, & \gamma_{2}^{1,1} = e_{0}e_{1}e_{4}e_{5}, \\\vspace{2mm}
\gamma_{3}^{0,0} = -e_{2}e_{3}e_{4}e_{5}, & \gamma_{3}^{1,0} = e_{2}e_{3}e_{4}e_{5}, & \gamma_{3}^{0,1} = -e_{0}e_{1}e_{6}e_{7}, & \gamma_{3}^{1,1} = e_{0}e_{1}e_{6}e_{7}, \\\vspace{2mm}
\gamma_{4}^{0,0} = e_{1}e_{3}e_{5}e_{7}, & \gamma_{4}^{1,0} = -e_{1}e_{3}e_{5}e_{7}, & \gamma_{4}^{0,1} = e_{0}e_{2}e_{4}e_{6}, & \gamma_{4}^{1,1} = -e_{0}e_{2}e_{4}e_{6}, \\\vspace{2mm}
\gamma_{5}^{0,0} = -e_{1}e_{3}e_{4}e_{6}, & \gamma_{5}^{1,0} = e_{1}e_{3}e_{4}e_{6}, & \gamma_{5}^{0,1} = -e_{0}e_{2}e_{5}e_{7}, & \gamma_{5}^{1,1} = e_{0}e_{2}e_{5}e_{7}, \\\vspace{2mm}
\gamma_{6}^{0,0} = -e_{1}e_{2}e_{5}e_{6}, & \gamma_{6}^{1,0} = e_{1}e_{2}e_{5}e_{6}, & \gamma_{6}^{0,1} = -e_{0}e_{3}e_{4}e_{7}, & \gamma_{6}^{1,1} = e_{0}e_{3}e_{4}e_{7}, \\\vspace{2mm}
\gamma_{7}^{0,0} = -e_{1}e_{2}e_{4}e_{7}, & \gamma_{7}^{1,0} = e_{1}e_{2}e_{4}e_{7}, & \gamma_{7}^{0,1} = -e_{0}e_{3}e_{5}e_{6}, & \gamma_{7}^{1,1} = e_{0}e_{3}e_{5}e_{6}. \\
\end{array}
\]
One also has
\[
\begin{array}{lllllllllll} \vspace{2mm}
\tau_{x_{1} - x_{2}} = \gamma_{1}^{0,1}, & \tau_{x_{1} + x_{2}} = \gamma_{1}^{1,1}, & \tau_{x_{3} - x_{4}} = \gamma_{1}^{0,0}, & \tau_{x_{3} + x_{4}} = \gamma_{1}^{1,0}, \\ \vspace{2mm}
\tau_{x_{1} - x_{3}} = \gamma_{2}^{0,1}, & \tau_{x_{1} + x_{3}} = \gamma_{2}^{1,1}, & \tau_{x_{2} - x_{4}} = \gamma_{2}^{0,0}, & \tau_{x_{2} + x_{4}} = \gamma_{2}^{1,0}, \\ \vspace{2mm}
\tau_{x_{1} - x_{4}} = \gamma_{3}^{0,1}, & \tau_{x_{1} + x_{4}} = \gamma_{3}^{1,1}, & \tau_{x_{2} - x_{3}} = \gamma_{3}^{0,0}, & \tau_{x_{2} + x_{3}} = \gamma_{3}^{1,0}. \\
\end{array}
\]






\section{Exceptional Lie groups}\label{ELg}

In this section, we recall the construction of the simply connected compact exceptional Lie groups $E_{8}, \ E_{7}, \ E_{6}$, and $F_{4}$ following \cite{Adams-L}.
We first summarize some basic properties of these groups.
Each of these groups is simply connected and simple, and theri dimensions are $248, 133, 78$, and $52$, respectively.
Let $G$ be one of these groups, and let $f$ be an involutive automorphism of $G$.
The fixed point subgroup $F(f,G)$ is connected and isomorphic to one of the groups listed in Table \ref{involutions}.
The isomorphism class of $F(f,G)$ can therefore be determined from $\dim F^{+}(f, \frak{g})$.

\begin{table}[h]
\centering
\caption{Fixed point subgroups of involutive automorphisms} \label{involutions}
\begin{tabular}{|c||c|c|c|c|c|c|cccccccccccc} \hline
Type & $G$ & $F(f, G)$ & $\dim F(f,G)$ & $\dim G/F(f,G)$ \\ \hline\hline
$FI$ & $F_{4}$ & $(Sp(1) \times Sp(3))/\mathbb{Z}_{2}$ & 24 & 28 \\
$FII$ & $F_{4}$ & $Spin(9)$ & 36 & 16 \\ \hline
$EI$ & $E_{6}$ & $Sp(4)/\mathbb{Z}_{2}$ & 36 & 42 \\
$EII$ & $E_{6}$ & $(Sp(1) \times SU(6))/\mathbb{Z}_{2}$ & 38 & 40 \\
$EIII$ & $E_{6}$ & $(U(1) \times Spin(10))/\mathbb{Z}_{4}$ & 46 & 32 \\
$EIV$ & $E_{6}$ & $F_{4}$ & 52 & 26 \\ \hline
$EV$ & $E_{7}$ & $SU(8)/\mathbb{Z}_{2}$ & 63 & 70 \\
$EVI$ & $E_{7}$ & $(Sp(1) \times Spin(12))/\mathbb{Z}_{2}$ & 69 & 64 \\
$EVII$ & $E_{7}$ & $(U(1) \times E_{6})/\mathbb{Z}_{3}$ & 79 & 54 \\ \hline
$EVIII$ & $E_{8}$ & $Spin(16)/\mathbb{Z}_{2}$ & 120 & 128 \\
$EIX$ & $E_{8}$ & $(Sp(1) \times E_{7})/\mathbb{Z}_{2}$ & 136 & 112 \\ \hline
\end{tabular}
\end{table}


\subsection{The Lie group $E_{8}$} \label{E8}

We begin with some basic properties of $Spin(16)$.
We realize $Spin(16)$ inside the Clifford algebra $Cl(\mathrm{span}_{\mathbb{R}}\{ e_{0}, \cdots, e_{15} \})$.
Let $Spin^{0}(8)$ and $Spin^{1}(8)$ be the subgroups defined inside 
$Cl(\mathrm{span}_{\mathbb{R}}\{ e_{0}, \cdots, e_{7} \})$ and $Cl(\mathrm{span}_{\mathbb{R}}\{ e_{8}, \cdots, e_{15} \})$, respectively.
Write $\frak{spin}^{i}(8)$ for the Lie algebra of $Spin^{i}(8)$ $(i=0,1)$.
We write $Spin^{0}(8) \cdot Spin^{1}(8)$ for the image of the product map $i : Spin^{0}(8) \times Spin^{1}(8) \to Spin(16)$.
For $i = 0,1$, let $T_{8}^{i}$ be the maximal torus in $Spin^{i}(8)$ given by
\[
\begin{split}
& T_{8}^{i} = \Big\{ T^{i}(\theta_{1}, \theta_{2}, \theta_{3}, \theta_{4}) = \prod_{k=1}^{4} \Big( \cos \frac{\theta_{k}}{2} + \sin \frac{\theta_{k}}{2} e_{8i+ 2k-2}e_{8i+2k-1} \Big) \ ;\ \theta_{k} \in \mathbb{R} \ (1 \leq k \leq 4) \Big\}. \\
\end{split}
\]
Its Lie algebra is
\[
\begin{split}
& \frak{t}_{8}^{i} = \left\{ t^{i}(\theta_{1}, \cdots, \theta_{4}) = \sum_{k=1}^{4}\frac{\theta_{k}}{2} e_{8i+2k-2}e_{8i+2k-1} \ ;\ \theta_{k} \in \mathbb{R} \ (1 \leq k \leq 4) \right\}.
\end{split}
\] 
Define
\[
\begin{split}
& T_{16} = \{ T_{16}(\theta_{1}, \cdots,\theta_{8}) = T^{0}_{8}(\theta_{1}, \cdots, \theta_{4}) T^{1}_{8}(\theta_{5}, \cdots, \theta_{8}) \ ;\ \theta_{k} \in \mathbb{R} \ (1 \leq k \leq 8) \}, \\
& \frak{t}_{16} = \frak{t}_{8}^{0} \oplus \frak{t}_{8}^{1} = \{ t_{16}(\theta_{1}, \cdots, \theta_{8}) = t_{8}^{0}(\theta_{1}, \cdots, \theta_{4}) + t_{8}^{1}(\theta_{5}, \cdots, \theta_{8}) \ ;\ \theta_{k} \in \mathbb{R} \ (1 \leq k \leq 8) \}.
\end{split}
\]
The subgroup $T_{16}$ is a maximal torus of $Spin(16)$ with Lie algebra $\frak{t}_{16}$.

For simplicity, write $\Delta = \Delta_{16}^{+}$ and $\Delta^{\mathbb{C}} = (\Delta_{16}^{+})^{\mathbb{C}}$.
Let $V$ and $V^{\mathbb{C}}$ be the representation spaces of $\Delta$ and $\Delta^{\mathbb{C}}$, respectively.
We write 
\[
\mathbb{O} \otimes^{+} \mathbb{O} := \mathbb{O}^{+} \otimes \mathbb{O}^{+}, \qquad
\mathbb{O} \otimes^{-} \mathbb{O} := \mathbb{O}^{-} \otimes \mathbb{O}^{-}.
\]
For $u,v \in \mathbb{O}$, we write the corresponding tensors by $u \otimes^{+} v$ or $u \otimes^{-} v$, respectively.
As a representation of $Spin^{0}(8) \times Spin^{1}(8)$, one has
\[
\Delta_{16}^{+}
\cong
(\Delta_{8}^{+} \otimes \Delta_{8}^{+})\ \oplus\ (\Delta_{8}^{-} \otimes \Delta_{8}^{-}).
\]
Accordingly,
\[
V = (\mathbb{O} \otimes^{+} \mathbb{O}) \oplus (\mathbb{O} \otimes^{-} \mathbb{O}), \quad
V^{\mathbb{C}} = (\mathbb{O} \otimes^{+} \mathbb{O})^{\mathbb{C}} \oplus (\mathbb{O} \otimes^{-} \mathbb{O})^{\mathbb{C}}.
\]
For $g_{i} \in Spin^{i}(8) \ (i = 0,1)$ and $u,v,w,z \in \mathbb{O}^{\mathbb{C}}$,
\[
\begin{split}
&\Delta^{\mathbb{C}}(i(g_{0}, g_{1}))(u \otimes^{+} v + w \otimes^{-} z) \\
& \hspace{10mm} = (\Delta_{8}^{+})^{\mathbb{C}}(g_{0})(u) \otimes^{+} (\Delta_{8}^{+})^{\mathbb{C}}(g_{1})(v) + (\Delta_{8}^{-})^{\mathbb{C}}(g_{0})(w) \otimes^{-} (\Delta_{8}^{-})^{\mathbb{C}}(g_{1})(z). \\
\end{split}
\]
The set of weights of $\Delta^{\mathbb{C}}$ with respect to $\frak{t}_{16}^{\mathbb{C}}$ is 
\[
\Pi_{16} = \left\{ \frac{1}{2} \sum_{i=1}^{8}\epsilon_{i}x_{i} \ ;\ \epsilon_{i} = \pm1, \ \epsilon_{1} \cdots \epsilon_{8} = 1 \right\}.
\]
For $\gamma \in \Pi_{16}$, the corresponding weight space $V_{\gamma}^{\mathbb{C}}$ is spanned by one of the following vectors
\[
(e_{2i} \pm i e_{2i+1}) \otimes^{+} (e_{2j} \pm i e_{2j+1}), \quad  (e_{2i} \pm i e_{2i+1}) \otimes^{-} (e_{2j} \pm i e_{2j+1}) \qquad (0 \leq i,j\le 3).
\]
In particular, $\overline{V_{\gamma}^{\mathbb{C}}}=V_{-\gamma}^{\mathbb{C}}$, where $\overline{ \ \cdot \ }$ denotes complex conjugation.


We recall the construction of the Lie algebra $\frak{e}_{8}$ from \cite{Adams-L}.
Let $L = \frak{spin}(16)$.
Fix a $Spin(16)$-invariant inner product $( \ ,\ )_{V}$ on $V$.
We also equip $L$ with the $Spin(16)$-invariant inner product $(\ ,\ )_{L}$ such that $(e_{r}e_{s}, e_{t}e_{u})_{L} = \delta_{rt}\delta_{su} \ (r < s, t < u)$.
Consider the direct sum $L \oplus V$.
Define a bilinear map $[\ ,\ ] : (L \oplus V) \times (L \oplus V) \rightarrow L \oplus V$ as follows.
For any $A, B \in L$ and $u, v \in V$, define
\[
[A,B] = [A,B]_{L}, \quad
[A, u] = \Delta(A)(u), 
\]
and define $[u,v] \in L$ by the condition $(C, [u,v])_{L} = (\Delta(C)(u), v)_{V}$ for any $C \in L$, where $[\ ,\ ]_{L}$ is the bracket of $L$.
This bilinear map is skew-symmetric and satisfies the Jacobi identity.
Thus, $L\oplus V$ becomes a Lie algebra, which we denote by $\frak{e}_{8}$.
The group of all automorphisms of $\frak{e}_{8}$ is the compact Lie group and called the exceptional compact Lie group $E_{8}$.
The Lie algebra of $E_{8}$ is $\frak{e}_{8}$.
It is well known that $E_{8}$ is simply connected and simple.
Furthermore, the center is trivial.
Define a map $\phi : Spin(16) \rightarrow E_{8}$ by
\[
\phi(g)(A + u) = \mathrm{Ad}(g)(A) + \Delta(g)(u) \qquad ( g \in Spin(16), A \in L, u \in V ).
\]
The $Spin(16)$-invariance of $(\ ,\ )_{L}$ and $(\ ,\ )_{V}$ implies that $\phi(g)$ preserves the bracket for every $g \in Spin(16)$.
Therefore,  $\phi$ is a homomorphism.
Since $\mathrm{Ker}\ \phi = \{ 1, e_{0}e_{1} \cdots e_{15} \}$, we obtain $\phi(Spin(16)) \cong Spin(16)/\mathbb{Z}_{2}$.
The homomorphism $\phi : L \rightarrow \frak{e}_{8}$ is injective, so we identify $L$ with its image in $\frak{e}_{8}$.
Under this identification,
\[
\frak{e}_{8} = L \oplus (\mathbb{O} \otimes^{+} \mathbb{O}) \oplus (\mathbb{O} \otimes^{-} \mathbb{O}).
\]
Let $p, q \in \mathbb{N}$ be even and satisfy $p + q = 16$.
Consider the homomorphism
\[
\psi_{p,q} = \phi \circ i_{p,q} : Spin(p) \times Spin(q) \rightarrow E_{8}.
\]
The kernel is $\mathrm{Ker}\ \psi_{p,q} = \{ (1,1), (-1,-1), (e_{0} \cdots e_{p-1}, e_{p} \cdots e_{15}), (-e_{0} \cdots e_{p-1}, -e_{p} \cdots e_{15}) \}$.
For $p = q = 8$, write $\psi = \psi_{8,8}$.
Its kernel is $\mathrm{Ker}\ \psi = \{ (\gamma_{0}^{a,b}, \gamma_{0}^{a,b}) \ ;\ a,b = 0,1 \}$.


We describe the root system of $\frak{e}_{8}$.
Set
\[
\begin{split}
& T = \Big\{ T(\theta_{1}, \cdots, \theta_{8}) = \phi(T_{16}(\theta_{1}, \cdots, \theta_{8})) \ ;\ \theta_{l} \in \mathbb{R} \ (1 \leq l \leq 8) \Big\}, \\
& \frak{t} = \Big\{ t(c_{1}, \cdots, c_{8}) = t_{16}(c_{1}, \cdots, c_{8}) \ ;\ c_{l} \in \mathbb{R} \ (1 \leq l \leq 8) \Big\}.
\end{split}
\]
The subgroup $T$ is a maximal torus of $E_{8}$ with Lie algebra $\frak{t}$.
In particular, $\mathrm{rank}\, E_{8} = 8$.
The root system of $\frak{e}_{8}^{\mathbb{C}}$ with respect to $\frak{t}^{\mathbb{C}}$ is
\[
\Sigma(\frak{e}_{8}, \frak{t}) 
= \Big\{ \pm x_{i} \pm x_{j} \ ;\ 1 \leq i < j \leq 8 \Big\} 
\cup \Big\{ \frac{1}{2} \sum_{i=1}^{8}\epsilon_{i} x_{i} \ ;\ \epsilon_{i} = \pm 1, \ \epsilon_{1} \cdots \epsilon_{8} = 1 \Big\}.
\]
For $\gamma \in \Sigma$, the corresponding root space is
\[
\frak{r}_{\gamma}^{\mathbb{C}} = 
\begin{cases}
\frak{r}_{\gamma}^{\mathbb{C}}(L, \frak{t}_{16}) & (\gamma \in \{ \pm x_{i} \pm x_{j} \ ;\ 1 \leq i < j \leq 8 \}), \\
V_{\gamma}^{\mathbb{C}} & (\text{otherwise}). \\
\end{cases}
\]
Choose a linear order on $\frak{t}$ for which the set of positive roots is 
\[
\Sigma^{+} = \{ x_{i} \pm x_{j} \ ;\ 1 \leq i < j \leq 8 \} \cup \Big\{ \frac{1}{2}\sum_{i=1}^{8}\epsilon_{i} x_{i} \in \Sigma \ ;\ \epsilon_{1} = 1 \Big\}.
\]
For $\gamma = \pm x_{i} \pm x_{j} \ (1 \leq i < j \leq 8)$,
\[
\frak{r}_{\gamma}(\frak{e}_{8}, \frak{t}) = \frak{r}_{\gamma}(L, \frak{t}_{16}).
\]
For $\gamma = (1/2)\sum_{i=1}^{8}\epsilon_{i}x_{i} \in \Sigma$, there exist $j,k \in \{ 0, 1,2,3 \}$ such that $\frak{r}_{\gamma}$ is one of
\[
\begin{split}
& \mathbb{R}( e_{2k} \otimes^{+} e_{2j} + e_{2k+1} \otimes^{+} e_{2j+1} ) \oplus \mathbb{R}( e_{2k} \otimes^{+} e_{2j+1} - e_{2k+1} \otimes^{+} e_{2j} ), \\
& \mathbb{R}( e_{2k} \otimes^{+} e_{2j} - e_{2k+1} \otimes^{+} e_{2j+1} ) \oplus \mathbb{R}( e_{2k} \otimes^{+} e_{2j+1} + e_{2k+1} \otimes^{+} e_{2j} ), \\
& \mathbb{R}( e_{2k} \otimes^{-} e_{2j} + e_{2k+1} \otimes^{-} e_{2j+1} ) \oplus \mathbb{R}( e_{2k} \otimes^{-} e_{2j+1} - e_{2k+1} \otimes^{-} e_{2j} ), \\
& \mathbb{R}( e_{2k} \otimes^{-} e_{2j} - e_{2k+1} \otimes^{-} e_{2j+1} ) \oplus \mathbb{R}( e_{2k} \otimes^{-} e_{2j+1} + e_{2k+1} \otimes^{-} e_{2j} ). \\
\end{split}
\]
For each $\alpha \in \Sigma$, let $SU^{\alpha}(2)$ be the connected subgroup of $E_{8}$ with Lie algebra $\frak{su}^{\alpha}(2)$.
If $\beta \in \Sigma_{\alpha, \pm1}$, let $SU^{\alpha, \beta}(3)$ be the connected subgroup with Lie algebra $\frak{su}^{\alpha, \beta}(3)$.
These subgroups satisfy $SU^{\alpha}(2) \cong SU(2)$ and $SU^{\alpha, \beta}(3) \cong SU(3)$.


\subsection{The Lie groups $E_{7}, E_{6}$, and $F_{4}$} \label{E7}

We now recall the construction of $E_{7}$.
Fix $\alpha \in \Sigma^{+}$.
The identity component of the centralizer $C(SU^{\alpha}(2), E_{8})$ is called the Lie group $E_{7}$, and we denote it by $E_{7}^{\alpha}$.
It is well known that the group $E_{7}^{\alpha}$ is a simply connected compact simple Lie group.
The center $C(E_{7}^{\alpha})$ is isomorphic to $\mathbb{Z}_{2}$ and $C(E_{7}^{\alpha}) = \{ e, \tau_{\alpha} \}$.
The Lie algebra of $E_{7}^{\alpha}$ is the centralizer $C(\frak{su}^{\alpha}(2), \frak{e}_{8})$ and is denoted by $\frak{e}_{7}^{\alpha}$.
We have
\[
\frak{e}_{7}^{\alpha} = \frak{t}^{\alpha} \oplus \bigoplus_{\beta \in \Sigma_{\alpha,0}^{+}}\frak{r}_{\beta}(\frak{e}_{8}, \frak{t})
\]
and the subalgebra $\frak{t}^{\alpha}$ is a maximal abelian subalgebra of $\frak{e}_{7}$.
Thus, $\mathrm{rank} \, E_{7} = 7$.
The root system of $(\frak{e}_{7}^{\alpha})^{\mathbb{C}}$ with respect to $(\frak{t}^{\alpha})^{\mathbb{C}}$ is $\Sigma(\frak{e}_{7}^{\alpha}, \frak{t}^{\alpha}) = \Sigma_{\alpha, 0}(\frak{e}_{8}, \frak{t})$.
For each $\beta \in \Sigma(\frak{e}_{7}^{\alpha}, \frak{t}^{\alpha})$, the corresponding root space $\frak{r}_{\beta}(\frak{e}_{7}^{\alpha}, \frak{t}^{\alpha})$ is given by $\frak{r}_{\beta}(\frak{e}_{8}, \frak{t})$.
For example, if $\alpha = x_{7} - x_{8}$, then
\[
\begin{split}
& \Sigma(\frak{e}_{7}^{\alpha}, \frak{t}^{\alpha}) = \left\{ \pm x_{i} \pm x_{j} \ ;\ 1 \leq i < j \leq 6 \right\} \cup \left\{ \pm(x_{7} + x_{8}) \right\} \cup \left\{ \frac{1}{2}\sum_{i=1}^{8}\epsilon_{i}x_{i} \in \Sigma \ ;\ \epsilon_{7} = \epsilon_{8} \right\}, \\
& \frak{e}_{7}^{\alpha} = \frak{su}^{\bar{\alpha}}(2) \oplus \frak{spin}^{\alpha}(12) \oplus (\mathbb{O} \otimes^{+} \mathbb{H}) \oplus (\mathbb{O} \otimes^{-} \mathbb{H}),
\end{split}
\]
where $\mathbb{H} = \mathrm{span}_{\mathbb{R}}\{ e_{0}, e_{1}, e_{2}, e_{3} \}$.

We next recall the group $E_{6}$.
Let $\alpha, \beta \in \Sigma^{+}(\frak{e}_{8}, \frak{t})$ satisfy $2(\alpha, \beta)/(\alpha,\alpha) = \pm1$.
Replacing $\beta$ by $\alpha-\beta$ if necessary, we may assume $2(\alpha, \beta)/(\alpha,\alpha) = -1$.
The identity component of the centralizer $C(SU^{\alpha, \beta}(3), E_{8})$ is called the Lie group $E_{6}$, and denoted by $E_{6}^{\alpha, \beta}$.
It is well known that the group $E_{6}^{\alpha, \beta}$ is a simply connected compact simple Lie group.
The center is isomorphic to $\mathbb{Z}_{3}$.
The Lie algebra of $E_{6}^{\alpha, \beta}$ is the centralizer $C(\frak{su}^{\alpha, \beta}(3), \frak{e}_{8})$ and is denoted by $\frak{e}_{6}^{\alpha,\beta}$.
We have
\[
\frak{e}_{6}^{\alpha, \beta} = \frak{t}^{\alpha, \beta} \oplus \bigoplus_{\gamma \in (\Sigma_{\alpha,0} \cap \Sigma_{\beta,0})^{+}} \frak{r}_{\gamma}(\frak{e}_{8}, \frak{t}).
\]
The subalgebra $\frak{t}^{\alpha, \beta}$ is a maximal abelian subalgebra of $\frak{e}_{6}^{\alpha, \beta}$.
Thus, $\mathrm{rank}\, E_{6}^{\alpha,\beta} = 6$.
The root system of $(\frak{e}_{6}^{\alpha,\beta})^{\mathbb{C}}$ with respect to $(\frak{t}^{\alpha,\beta})^{\mathbb{C}}$ is $\Sigma(\frak{e}_{6}^{\alpha, \beta}, \frak{t}^{\alpha,\beta}) = \Sigma_{\alpha, 0} \cap \Sigma_{\beta,0}$.
For each $\gamma \in \Sigma(\frak{e}_{6}^{\alpha,\beta}, \frak{t}^{\alpha,\beta})$, the corresponding root space $\frak{r}_{\gamma}(\frak{e}_{6}^{\alpha, \beta}, \frak{t}^{\alpha, \beta})$ is $\frak{r}_{\gamma}(\frak{e}_{8}, \frak{t})$.
For example, if $\alpha = x_{7} - x_{8}$ and $\beta = x_{6} - x_{7}$, then
\[
\begin{split}
& \Sigma(\frak{e}_{6}^{\alpha, \beta}, \frak{t}^{\alpha, \beta}) 
= \{ \pm x_{i} \pm x_{j} \ ;\ 1 \leq i < j \leq 5 \} 
\cup \left\{ \frac{1}{2} \sum_{i=1}^{8}\epsilon_{i}x_{i} \in \Sigma \ ;\ \epsilon_{6} = \epsilon_{7} = \epsilon_{8} \right\}, \\
& \frak{e}_{6}^{\alpha, \beta} = \mathbb{R}t(0, \cdots, 0, 1,1,1) \oplus \frak{spin}^{\alpha, \beta}(10) \oplus (\mathbb{O} \otimes^{+} \mathbb{C}) \oplus (\mathbb{O} \otimes^{-} \mathbb{C}),
\end{split}
\]
where $\mathbb{C} = \mathrm{span}_{\mathbb{R}}\{ e_{0}, e_{1} \}$.

We next recall the group $F_{4}$.
Since $G_{2} \subset Spin(8)$, we may regard $G_{2}$ as a subgroup of $G_{2}$ of $E_{8}$.
Set $G_{2} = \psi(G_{2} \times \{ e \})$ and $G'_{2} = \psi(\{ e \} \times G_{2})$.
The identity component of the centralizer $C(G'_{2}, E_{8})$ is called the exceptional Lie group $F_{4}$.
It is well known that $F_{4}$ is a simply connected compact simple Lie group.
The Lie algebra of $F_{4}$ is the centralizer $C(\frak{g}'_{2}, \frak{e}_{8})$ and is denoted by $\frak{f}_{4}$.
Since $G'_{2}$ fixes $1 \in \mathbb{O}$ and acts transitively on the unit sphere in $\mathrm{Im}\mathbb{O}$, the fixed-point subspace is $\mathbb{R}e_{0}$. 
Hence
\[
C(\frak{g}'_{2}, V) = (\mathbb{O} \otimes^{+} \mathbb{R}e_{0}) \oplus (\mathbb{O} \otimes^{-} \mathbb{R}e_{0}).
\]
We write $\mathbb{O} \otimes^{+} \mathbb{R}e_{0}$ and $\mathbb{O} \otimes^{-} \mathbb{R}e_{0}$ by $\mathbb{O}^{+}$ and $\mathbb{O}^{-}$.
The identity component of $C(G'_{2}, \psi(Spin(16)))$ is the subgroup $Spin(9)$ arising from $Cl(\mathrm{span}_{\mathbb{R}}\{ e_{0}, \cdots, e_{8} \})$.
Hence,
\[
\frak{f}_{4} = \frak{spin}(9) \oplus \mathbb{O}^{+} \oplus \mathbb{O}^{-}.
\]






\section{Conjugation orbits and exceptional symmetric spaces}

In this section, we prove the main results of the paper.
Fix $\alpha = x_{7} - x_{8}$ and $\beta = x_{6} - x_{7}$, and write $E_{7} = E_{7}^{\alpha}$ and $E_{6} = E_{6}^{\alpha, \beta}$.
Since $SU^{\alpha}(2) \subset SU^{\alpha, \beta}(3) \subset G'_{2}$, we have the chain of inclusions
\[
G_{2} \subset F_{4} \subset E_{6} \subset E_{7} \subset E_{8}.
\]
Fix a bi-invariant metric on $E_{8}$, and regard these compact exceptional Lie groups as compact symmetric spaces with the induced metrics.

\subsection{Maximal antipodal subgroups of $E_{8}$}

We recall Adams's classification of the maximal antipodal subgroups of $E_{8}$ \cite{Adams}.
Set
\[
A(G_{2}) = \{ \gamma_{i} \ ;\ 0 \leq i \leq 7\}, \qquad
A(Spin(8)) = \{ \gamma_{i}^{a,b} \ ;\ 0 \leq i \leq 7, a,b = 0,1 \}.
\]
The subgroups $A(G_{2})$ and $A(Spin(8))$ are maximal antipodal subgroups of $G_{2}$ and $Spin(8)$ of order $8$ and $32$, respectively \cite{Tanaka-Tasaki-Yasukura, Wood}.
Using these subgroups, define
\[
\begin{split}
A(E_{8}) &= \psi \big( A(Spin(8)) \times A(Spin(8)) \big) = \left\{ \psi( \gamma_{i}^{0,0}, \gamma_{j}^{a,b}) \ ;\ 0 \leq i,j \leq 7, a,b = 0,1 \right\}.
\end{split}
\]
The subgroup $A(E_{8})$ is an antipodal subgroup of $E_{8}$ of order $256$.
Let $T\subset E_{8}$ be the maximal torus fixed in Subsection 3.1.
Set $A(T) = \{ t \in T \ ;\ t^{2} = e \}$.
This is the maximal antipodal subgroup of $T$ of order $256$.
In $E_{8}$, there exists an involutive element $x$ such that $xtx^{-1} = t^{-1}$ for any $t \in T$.
Then,
\[
B(E_{8}) = A(T) \cup x(A(T))
\]
is an antipodal subgroup of $E_{8}$ of order $512$.

\begin{thm} \cite{Adams}
The subgroups $A(E_{8})$ and $B(E_{8})$ are maximal antipodal subgroups of $E_{8}$ of order $256$ and $512$, respectively.
Every maximal antipodal subgroup is conjugate to either $A(E_{8})$ or $B(E_{8})$.
\end{thm}

Since the elements of $A(G_{2})$ are constructed using $\mathbb{O}$, the construction of $A(E_{8})$ is based on $\mathbb{O}$.
On the other hand, $B(E_{8})$ is constructed using a maximal torus.
In the present article, we consider the maximal antipodal subgroup $A(E_{8})$.

\begin{thm} \label{main-1}
Let $G \in \{ G_{2}, F_{4}, E_{6}, E_{7}, E_{8} \}$ and $p \in A(E_{8})$.
Then, $p G p^{-1} = G$ and $\theta_{p} : G \to G \ ;\ p \mapsto gpg^{-1}$ is an involutive automorphism.
The conjugation orbit $\mathcal{O}_{G}(p)$ is a simply connected compact symmetric space.
More precisely, it is isomorphic to one of $FI, FII, EI, \cdots, EIX$, or $G_{2}/SO(4)$.
Moreover, $\mathcal{O}_{G}(p)$ is a totally geodesic submanifold of $E_{8}$.
\end{thm}

\begin{proof}
For every $p \in A(E_{8})$, the adjoint action $\mathrm{Ad}(p)$ preserves each of $\frak{su}^{\alpha}(2), \frak{su}^{\alpha, \beta}(3),$ $\frak{g}_{2}'$.
By the definition of $E_{7}, E_{6}, F_{4}$, the adjoint action $\mathrm{Ad}(p)$ preserves $\frak{e}_{7}, \frak{e}_{6}, \frak{f}_{4}$, respectively.
Hence, $\theta_{p}$ restricts to an involutive automorphism of each of $E_{7}, E_{6}, F_{4}$.
On the other hand, by the construction of $A(E_{8})$, we also have $\mathrm{Ad}(p)(\frak{g}_{2}) \subset \frak{g}_{2}$, and $\theta_{p}$ restricts to an involutive automorphism of $G_{2}$.
Since each $G \in \{ G_{2}, F_{4}, E_{6}, E_{7}, E_{8}\}$ is simply connected, the subgroup $F(\theta_{p}, G)$ is connected.
Since $\mathcal{O}_{G}(p) \cong G/F(\theta_{p}, G)$, the orbit $\mathcal{O}_{G}(p)$ is a simply connected compact symmetric space.
Moreover, 
\[
\mathrm{exp} \bigl( F^{-}(\mathrm{Ad}(p), \frak{g}) \bigr) = \bigcup_{g \in G}g p g^{-1} p^{-1} = \mathcal{O}_{G}(p)p^{-1}
\]
is a totally geodesic submanifold of $G$.
Since $G$ is a totally geodesic submanifold of $E_{8}$, the orbit $\mathcal{O}_{G}(p)$ is a totally geodesic submanifold of $E_{8}$.
\end{proof}

\begin{remark}
Theorem \ref{main-1} does not hold with $A(E_{8})$ replaced by $B(E_{8})$.
In fact, for each $p \in B(E_{8})$, the automorphism $\theta_{p}$ preserves $E_{7}$ and $E_{6}$.
On the other hand, if $\gamma = \frac{1}{2}(x_{1} + \cdots + x_{8})$, then $\theta_{\tau_{\gamma}}$ does not preserve $F_{4}$.
\end{remark}

In Subsections \ref{E_{8}-orbit}, $\cdots$, \ref{G_{2}-orbit}, we will study $\mathcal{O}_{G}(p)$ for each $G \in \{ G_{2}, F_{4}, E_{6}, E_{7}, E_{8}\}$ and $p \in A(E_{8})$.
We introduce the following notation for later use.
For $1 \leq k \leq 7$, set
\[
p_{k} = \psi( \gamma_{1}^{0,0}, \gamma_{k}^{0,0}), \quad
q_{k} = \psi( \gamma_{0}^{0,0}, \gamma_{k}^{1,0}), \quad
r_{k} = \psi( \gamma_{0}^{0,0}, \gamma_{k}^{0,0}). 
\]
Also set
\[
p_{0} = \psi( \gamma_{0}^{0,0}, \gamma_{0}^{1,0} ), \quad
q_{0} = \psi( \gamma_{1}^{0,0}, \gamma_{0}^{0,0} ), \quad
r_{0} = \psi( \gamma_{0}^{0,0}, \gamma_{0}^{0,0}). 
\]
Since $r_{0} = e$ in $E_{8}$, we write $e$ by $r_{0}$ when we consider $e$ as an element of $A(E_{8})$.


\subsection{The $E_{8}$-orbits} \label{E_{8}-orbit}

We study $\mathcal{O}_{E_{8}}(p)$ for $p \in A(E_{8})$.
We first recall the polars of $E_{8}$.
According to the classification of polars by Chen and Nagano \cite{Chen-Nagano1}, $E_{8}$ has two nontrivial polars, isometric to $EVIII$ and $EIX$, respectively.
Let $EVIII_{+}$ and $EIX_{+}$ denote the polars of $e$ isometric to $EVIII$ and $EIX$, respectively.

\begin{lemm} \label{E_{8}}
Set $A(EVIII_{+}) = A(E_{8}) \cap EVIII_{+}$ and $A(EIX_{+}) = A(E_{8}) \cap EIX_{+}$.
Then,
\[
\begin{split}
& A(EVIII_{+}) = \Big\{ \psi( \gamma_{i}^{0,0}, \gamma_{j}^{a,b} ) \ ;\ 1 \leq i, j \leq 7, a,b = 0,1 \Big\} \cup \Big\{ \psi(\gamma_{0}^{0,0}, \gamma_{0}^{a,b}) \ ;\ (a,b) \not= (0,0) \Big\}, \\
& A(EIX_{+}) = \Big\{ \psi( \gamma_{0}^{0,0}, \gamma_{i}^{a,b} ), \psi( \gamma_{i}^{0,0}, \gamma_{0}^{a,b}) \ ;\ 1 \leq i \leq 7, a,b = 0,1 \Big\}.
\end{split}
\]
For every $p \in A(E_{8})$, the orbit $\mathcal{O}_{E_{8}}(p)$ is one of $\{ r_{0} \}, EVIII_{+}$, and $EIX_{+}$.
\end{lemm}

\begin{proof}
Since every element of $A(E_{8})$ is an involution, we have $A(E_{8}) \subset F(s_{e}, E_{8})$.
Therefore, for every $p\in A(E_{8}) \setminus \{ r_{0} \}$, the orbit $\mathcal{O}_{E_{8}}(p)$ is a polar of $e$ in $E_{8}$.
The dimensions of these polars are $\dim EVIII_{+} = 128$ and $\dim EIX_{+} = 112$.
Consequently,
\[
\begin{split}
& A(EVIII_{+}) = \{ p \in A(E_{8}) \ ;\ \dim F^{+}(\mathrm{Ad}(p), \frak{e}_{8}) = 120 \}, \\
& A(EIX_{+}) = \{ p \in A(E_{8}) \ ;\ \dim F^{+}(\mathrm{Ad}(p), \frak{e}_{8}) = 136 \}. 
\end{split}
\]
Let $p = \psi(a,b) \in A(E_{8})$.
Since $\frak{e}_{8} = L \oplus V$,
\[
F^{+}(\mathrm{Ad}(p), \frak{e}_{8}) = F^{+}(\mathrm{Ad}(p)|_{L}, L) \oplus F^{+}(\mathrm{Ad}(p)|_{V}, V).
\]
On the $V$-component, we have
\[
\begin{split}
& F^{+}(\mathrm{Ad}(p)|_{V}, V) = F^{+}(\Delta(\psi(a,b)), V) \\
&= \bigl( F^{+}(\Delta_{8}^{+}(a), \mathbb{O}^{+}) \otimes F^{+}(\Delta_{8}^{+}(b), \mathbb{O}^{+}) \bigr)
\oplus \bigl( F^{-}(\Delta_{8}^{+}(a), \mathbb{O}^{+}) \otimes F^{-}(\Delta_{8}^{+}(b), \mathbb{O}^{+}) \bigr) \\
& \quad \quad \oplus \bigl( F^{+}(\Delta_{8}^{-}(a), \mathbb{O}^{-}) \otimes F^{+}(\Delta_{8}^{-}(b), \mathbb{O}^{-}) \bigr)
\oplus \bigl( F^{-}(\Delta_{8}^{-}(a), \mathbb{O}^{-}) \otimes F^{-}(\Delta_{8}^{-}(b), \mathbb{O}^{-}) \bigr). \\
\end{split}
\]
For each $\gamma_{i}^{a,b} \in A(Spin(8))$,
\[
\begin{array}{lll} \vspace{1mm}
\dim F^{+}(\Delta_{8}^{+}(\gamma_{0}^{0,0}), \mathbb{O}^{+}) = \dim F^{+}(\Delta_{8}^{-}(\gamma_{0}^{0,0}), \mathbb{O}^{-}) = 8, \\ \vspace{1mm}
\dim F^{-}(\Delta_{8}^{+}(\gamma_{0}^{0,0}), \mathbb{O}^{+}) = \dim F^{-}(\Delta_{8}^{-}(\gamma_{0}^{0,0}), \mathbb{O}^{-}) = 0, \\ \vspace{1mm}
\dim F^{-}(\Delta_{8}^{+}(\gamma_{0}^{1,0}), \mathbb{O}^{+}) = \dim F^{-}(\Delta_{8}^{-}(\gamma_{0}^{1,0}), \mathbb{O}^{-}) = 8, \\ \vspace{1mm}
\dim F^{+}(\Delta_{8}^{+}(\gamma_{0}^{1,0}), \mathbb{O}^{+}) = \dim F^{+}(\Delta_{8}^{-}(\gamma_{0}^{1,0}), \mathbb{O}^{-}) = 0, \\ \vspace{1mm}
\dim F^{+}(\Delta_{8}^{+}(\gamma_{0}^{0,1}), \mathbb{O}^{+}) = \dim F^{-}(\Delta_{8}^{-}(\gamma_{0}^{0,1}), \mathbb{O}^{-}) = 8, \\ \vspace{1mm}
\dim F^{-}(\Delta_{8}^{+}(\gamma_{0}^{0,1}), \mathbb{O}^{+}) = \dim F^{+}(\Delta_{8}^{-}(\gamma_{0}^{0,1}), \mathbb{O}^{-}) = 0, \\ \vspace{1mm}
\dim F^{-}(\Delta_{8}^{+}(\gamma_{0}^{1,1}), \mathbb{O}^{+}) = \dim F^{+}(\Delta_{8}^{-}(\gamma_{0}^{1,1}), \mathbb{O}^{-}) = 8, \\ \vspace{1mm}
\dim F^{+}(\Delta_{8}^{+}(\gamma_{0}^{1,1}), \mathbb{O}^{+}) = \dim F^{-}(\Delta_{8}^{-}(\gamma_{0}^{1,1}), \mathbb{O}^{-}) = 0, \\
\dim F^{\pm}(\Delta_{8}^{\pm}(\gamma_{i}^{a,b}), \mathbb{O}^{\pm}) = 4 \ \  (1 \leq i \leq 7, a,b = 0,1).
\end{array}
\]
On the $L$-component, the fixed point subalgebra is given, up to isomorphism, by
\[
\begin{split}
&F^{+}( \mathrm{Ad}(\psi(\gamma_{i}^{0,0}, \gamma_{j}^{a,b}))|_{L}, L) \\
& \quad\quad \cong
\begin{cases}
\frak{spin}(8) \oplus \frak{spin}(8) & (i,j \not= 0, \ a,b = 0,1 \ \text{or} \ i = j = 0, \ (a,b) \in \{ (0,1), (1,1) \}), \\
\frak{spin}(12) \oplus \frak{spin}(4) & (\text{if exactly one of $i$ and $j$ is zero}.), \\
\frak{spin}(16) & (i = j = 0, (a,b) \in \{ (0,0), (1,0) \}). \\
\end{cases}
\end{split}
\]
In particular, $\dim F^{+}( \mathrm{Ad}(p)|_{L}, L)$ takes one of the values $56, 72, 120$.
Combining the dimensions of the fixed point subspaces in the $L$- and $V$-components yields the asserted descriptions of $A(EVIII_{+})$ and $A(EIX_{+})$.
\end{proof}

In particular,
\[
EVIII_{+} = \mathcal{O}_{E_{8}}(p_{0}), \qquad
EIX_{+} = \mathcal{O}_{E_{8}}(q_{0}).
\]


\subsection{The $E_{7}$-orbits} \label{E_{7}-orbit}

We consider $\mathcal{O}_{E_{7}}(p)$ for $p \in A(E_{8})$.
We first determine the intersection $A(E_{7}) = A(E_{8}) \cap E_{7}$.
The subgroup
\[
A(E_{8}) \cap C(SU^{\alpha}(2), E_{8}) = \big\{ \psi(\gamma_{i}^{0,0}, \gamma_{0}^{a,b}), \psi(\gamma_{i}^{0,0}, \gamma_{1}^{a,b}) \ ; \ 0 \leq i \leq 7, a,b = 0,1 \big\}
\]
is contained in a subgroup $\phi(Spin(12)) \cong Spin(12)$ of $E_{7}$.
Therefore, we have
\[
A(E_{7}) = A(E_{8}) \cap C(SU^{\alpha}(2), E_{8}).
\]

According to the classification of polars \cite{Chen-Nagano1}, 
$E_{7}$ has three polars: two are isometric to $EVI$, and the third is a pole.
Let $EVI_{+}$ and $EVI_{+}'$ denote the polars of $e$ isometric to $EVI$ contained in $EVIII_{+}$ and $EIX_{+}$, respectively.
The nontrivial pole of $e$ is $r_{1}$.
Thus,
\[
EVI_{+} = E_{7} \cap EVIII_{+}, \qquad
EVI_{+}' \cup \{ r_{1} \} = E_{7} \cap EIX_{+}.
\]

\begin{lemm} \label{EVI_{+}}
Set $A(EVI_{+}) = A(E_{8}) \cap EVI_{+}$ and $A(EVI'_{+}) = A(E_{8}) \cap EVI'_{+}$.
Then,
\[
\begin{split}
A(EVI_{+}) &= \big\{ \psi(\gamma_{i}^{0,0}, \gamma_{1}^{a,b}) \ ; \ 1 \leq i \leq 7, a,b = 0,1 \big\} \cup \big\{ \psi(\gamma_{0}^{0,0}, \gamma_{0}^{a,b}) \ ;\ (a,b) \not= (0,0) \big\}, \\
A(EVI_{+}') &= r_{1}A(EVI_{+}) \\ 
&= \big\{ \psi(\gamma_{i}^{0,0}, \gamma_{0}^{a,b}) \ ; \ 1 \leq i \leq 7, a,b = 0,1 \big\} \cup \big\{ \psi(\gamma_{0}^{0,0}, \gamma_{1}^{a,b}) \ ;\ (a,b) \not= (0,0) \big\} 
\end{split}
\]
For every $p \in A(E_{7})$, the orbit $\mathcal{O}_{E_{7}}(p)$ is one of $\{ r_{0} \}, \{ r_{1} \}, EVI_{+}$, and $EVI_{+}'$.
\end{lemm}

\begin{proof}
The first identity follows from $A(EVI_{+})=A(E_{7})\cap A(EVIII_{+})$.
Moreover, $A(EVI'_{+}) \cup \{ r_{1} \} = A(E_{7}) \cap A(EIX_{+})$.
Therefore, the asserted descriptions follow from Lemma \ref{E_{8}}.
\end{proof}

Note that
\[
EVI_{+} = \mathcal{O}_{E_{7}}(p_{0}), \quad
EVI'_{+} = \mathcal{O}_{E_{7}}(q_{0}).
\]
To determine the remaining $E_{7}$-orbits, we divide the argument into two parts.
We first consider the orbits through the points of $A(EVIII_{+}) \setminus A(E_{7})$, and then those through the points of $A(EIX_{+}) \setminus A(E_{7})$.


For $1 \leq i \leq 3$, the element $p_{2i} \in A(EVIII_{+}) \setminus A(E_{7})$ satisfies
\[
\begin{split}
F^{+}(\mathrm{Ad}(p_{2i})|_{\frak{e}_{7} \cap V}, \frak{e}_{7} \cap L) & \cong \frak{spin}(6) \oplus \frak{spin}(6) \oplus \mathbb{R}, \\
F^{+}(\mathrm{Ad}(p_{2i})|_{\frak{e}_{7} \cap V}, \frak{e}_{7} \cap V) & = F^{+}(\Delta(p_{2i}), (\mathbb{O} \otimes^{+} \mathbb{H}) \oplus^{-} (\mathbb{O} \otimes \mathbb{H}) ) \\
& = (\mathbb{H} \otimes^{+} \mathbb{C}_{i})
\oplus (\mathbb{H}^{\perp} \otimes^{+} \mathbb{C}_{i}^{\perp}) 
\oplus (\mathbb{H} \otimes^{-} \mathbb{C}_{i}) 
\oplus (\mathbb{H}^{\perp} \otimes^{-} \mathbb{C}_{i}^{\perp}),
\end{split}
\]
where $\mathbb{C}_{i} = \mathbb{R} \oplus \mathbb{R}e_{i}$, the subspace $\mathbb{C}_{i}^{\perp}$ is the orthogonal complement of $\mathbb{C}_{i}$ in $\mathbb{H}$, and $\mathbb{H}^{\perp}$ is the orthogonal complement of $\mathbb{H}$ in $\mathbb{O}$.
Hence, $\dim F^{+}(\mathrm{Ad}(p_{2i}), \frak{e}_{7}) = 63$ and $\mathcal{O}_{E_{7}}(p_{2i})$ is isometric to $EV$.
For $1\leq i\leq3$, set
\[
EV_{i} = \mathcal{O}_{E_{7}}(p_{2i}).
\]

\begin{lemm} \label{EV_{i}}
For $1\leq i\leq3$, set $A(EV_{i}) = A(E_{8}) \cap EV_{i}$.
Then,
\[
A(EV_{i}) = \big\{ \psi(\gamma_{k}^{0,0}, \gamma_{2i}^{a,b}), \psi(\gamma_{k}^{0,0}, \gamma_{2i+1}^{a,b}) \ ;\ 1 \leq k \leq 7, a,b = 0,1  \big\}.
\]
For every $p \in A(EVIII_{+}) \setminus A(E_{7})$, the orbit $\mathcal{O}_{E_{7}}(p)$ is one of $EV_{1},EV_{2}$, and $EV_{3}$.
\end{lemm}

\begin{proof}
Let $ 1 \leq i \not= j \leq 3$.
Then, 
\[
\begin{split}
& \mathrm{Ad}(p_{2i})|_{\frak{su}^{\alpha}(2)} \not= \mathrm{Ad}(p_{2j})|_{\frak{su}^{\alpha}(2)}, \qquad
\mathrm{Ad}(p_{2i})|_{\frak{su}^{\alpha}(2)} = \mathrm{Ad}(p_{2i+1})|_{\frak{su}^{\alpha}(2)}. \\
\end{split}
\]
Since $E_{7}$ is the identity component of $C(SU^\alpha(2),E_8)$, the restriction of $\mathrm{Ad}(p)$ to $\frak{su}^{\alpha}(2)$ is constant along each $E_{7}$-orbit.
It follows that
\[
A(EV_{i}) \subset \{ \psi(\gamma_k^{0,0},\gamma_{2i}^{a,b}), \psi(\gamma_k^{0,0},\gamma_{2i+1}^{a,b}) \ ;\ 1 \leq k \leq 7, \ a,b = 0,1 \}.
\]
In $Spin(8)$, 
\[
F(s_{e}, Spin(8)) = \{ \gamma_{0}^{a,b} \ ;\ a,b = 0,1 \} \sqcup \mathcal{O}_{Spin(8)}(\gamma_{1}^{0,0}),
\]
where $\mathcal{O}_{Spin(8)}(\gamma_{1}^{0,0})$ is isometric to $\tilde{G}_{4}(\mathbb{R}^{8})$, which is the oriented Grassmannian of oriented $4$-planes of $\mathbb{R}^{8}$.
Consequently, $\{ \gamma_{k}^{a,b} \ ;\ 1 \leq k \leq 7, a,b = 0, 1 \} \subset \mathcal{O}_{Spin(8)}(\gamma_{1}^{0,0})$.
Since $\psi(Spin^{0}(8) \times \{ e \}) \subset E_{7}$, it follows that $\{ \psi(\gamma_{k}^{0,0}, \gamma_{2i}^{a,b}) \ ;\ 1 \leq k \leq 7, a,b = 0,1 \}$ is contained in $A(EV_{i})$.
Consider the subgroup $K = SU^{x_{5} + x_{6}}(2) \times SU^{x_{5} - x_{6}}(2) \times SU^{x_{7} + x_{8}}(2) \subset E_{7}$.
For each $1 \leq i \leq 3$, we have 
\[
\big\{ \psi(\gamma_{1}^{0,0}, \gamma_{2i}^{a,b}), \psi(\gamma_{1}^{0,0}, \gamma_{2i+1}^{a,b}) \ ;\ a,b = 0,1 \big\} \subset \mathcal{O}_{K}( \psi(\gamma_{1}^{0,0}, \gamma_{2i}^{0,0})).
\]
It follows that every element in the displayed set belongs to $A(EV_{i})$.
Therefore, the asserted equality holds.
Finally, 
\[
A(EVIII_{+}) \setminus A(E_{7}) = \bigsqcup_{1 \leq i \leq 3}A(EV_{i}),
\]
which completes the proof.
\end{proof}


For any $1 \leq i \leq 3$, the element $q_{2i} \in A(EIX_{+}) \setminus A(E_{7})$ satisfies
\[
\begin{split}
F^{+}(\mathrm{Ad}(q_{2i})|_{\frak{e}_{7} \cap V}, \frak{e}_{7} \cap L) & \cong \frak{spin}(10) \oplus \frak{spin}(2) \oplus \mathbb{R}, \\
F^{+}(\mathrm{Ad}(p_{2i})|_{\frak{e}_{7} \cap V}, \frak{e}_{7} \cap V) & = F^{+}(\Delta(q_{2i}), \mathbb{O} \otimes^{+} \mathbb{H} \oplus \mathbb{O} \otimes^{-} \mathbb{H}) 
= (\mathbb{O} \otimes^{+} \mathbb{C}_{i}^{\perp}) \oplus (\mathbb{O} \otimes^{-} \mathbb{C}_{i}^{\perp}). \\
\end{split}
\]
It follows that $\dim F^{+}(\mathrm{Ad}(q_{2i}), \frak{e}_{7}) = 79$.
Hence, $\mathcal{O}_{E_{7}}(q_{2i})$ is isometric to $EVII$.
For $1\leq i\leq3$, set
\[
EVII_{i} = \mathcal{O}_{E_{7}}(q_{2i}).
\]

\begin{lemm} \label{EVII_{i}}
For $1\leq i\leq3$, set $A(EVII_{i}) = A(E_{8}) \cap EVII_{i}$.
Then,
\[
A(EVII_{i}) = \big\{ \psi(\gamma_{0}^{0,0}, \gamma_{2i}^{a,b}), \psi(\gamma_{0}^{0,0}, \gamma_{2i+1}^{a,b}) \ ; \ a,b = 0, 1 \big\}.
\]
For every $p \in A(EIX_{+}) \setminus A(E_{7})$, the orbit $\mathcal{O}_{E_{7}}(p)$ is one of $EVII_{1}, EVII_{2}$, and $EVII_{3}$.
\end{lemm}

\begin{proof}
As in the proof of Lemma \ref{EV_{i}}, $A(EVII_{i})$ is contained in the set displayed in the statement.
For the reverse inclusion, for each $1 \leq i \leq 3$, we have 
\[
\big\{ \psi(\gamma_{0}^{0,0}, \gamma_{2i}^{a,b}), \psi(\gamma_{0}^{0,0}, \gamma_{2i+1}^{a,b}) \ ;\ a,b = 0,1 \big\} \subset \mathcal{O}_{K}(\psi( \gamma_{0}^{0,0}, \gamma_{2i}^{0,0})),
\]
where $K$ is a subgroup of $E_{7}$ defined in Lemma \ref{EV_{i}}.
Thus, the set displayed in the statement is contained in $A(EVII_{i})$.
Together with the first inclusion, this proves the equality.
Finally,
\[
A(EIX_{+}) \setminus A(E_{7}) = \bigsqcup_{1 \leq i \leq 3}A(EVII_{i}),
\]
which completes the proof.
\end{proof}

Lemmas \ref{EVI_{+}}, \ref{EV_{i}}, and \ref{EVII_{i}} yield the following classification.

\begin{lemm} \label{E_{7}}
For every $p \in A(E_{8})$, the orbit $\mathcal{O}_{E_{7}}(p)$ is one of
\[
\{ r_{0} \}, \quad 
\{ r_{1} \}, \quad
EV_{i}, \quad
EVI_{+}, \quad
EVI_{+}', \quad
EVII_{i}
\qquad (1 \leq i \leq 3),
\]
where
\[
EV_{i} = \mathcal{O}_{E_{7}}(p_{2i}), \quad
EVI_{+} = \mathcal{O}_{E_{7}}(p_{0}), \quad
EVI_{+}' = \mathcal{O}_{E_{7}}(q_{0}), \quad
EVII_{i} = \mathcal{O}_{E_{7}}(q_{2i}).
\]
Furthermore,
\[
EVI_{+} \cup \bigcup_{i=1}^{3}EV_{i} \subset EVIII_{+}, \qquad
\{ r_{1} \} \cup EVI_{+}' \cup \bigcup_{i=1}^{3}EVII_{i} \subset EIX_{+}.
\]
\end{lemm}


\subsection{The $E_{6}$-orbits} \label{E_{6}-orbit}

We consider $\mathcal{O}_{E_{6}}(p)$ for $p \in A(E_{8})$.
We first determine the intersection $A(E_{6}) = A(E_{8}) \cap E_{6}$.
The subgroup
\[
A(E_{8}) \cap C(SU^{\alpha, \beta}(3), E_{8}) = \big\{ \psi(\gamma_{i}^{0,0}, \gamma_{0}^{a,b}) \ ; \ 0 \leq i \leq 7, a,b = 0,1 \big\}
\]
is contained in the subgroup $\phi(Spin^{0}(8)) \cong Spin(8)$ of $E_{6}$.
Therefore, we have
\[
A(E_{6}) = A(E_{8}) \cap C(SU^{\alpha, \beta}(3), E_{8}).
\]

According to the classification of polars \cite{Chen-Nagano1}, $E_{6}$ has two polars, isometric to $EII$ and $EIII$, respectively.
Let $EII_{+}$ and $EIII_{+}$ denote the polars of $e$ isometric to $EII$ and $EIII$, respectively.

\begin{lemm} \label{EII_{+}EIII_{+}}
Set $A(EII_{+}) = A(E_{8}) \cap EII_{+}$ and $A(EIII_{+}) = A(E_{8}) \cap EIII_{+}$.
Then,
\[
\begin{split}
& A(EII_{+}) = \big\{ \psi(\gamma_{k}^{0,0}, \gamma_{0}^{a,b}) \ ;\ 1 \leq k \leq 7, a,b = 0,1 \big\}, \\
& A(EIII_{+}) = \big\{ \psi(\gamma_{0}^{0,0}, \gamma_{0}^{a,b}) \ ;\ (a,b) \not= (0,0) \big\}.
\end{split}
\]
For every $p \in A(E_{6})$, the orbit $\mathcal{O}_{E_{6}}(p)$ is one of $\{ r_{0} \}, EII_{+}$, and $EIII_{+}$.
\end{lemm}

\begin{proof}
The two intersections $E_{6} \cap EVIII_{+}$ and $E_{6} \cap EIX_{+}$ are the polars $EII_{+}$ and $EIII_{+}$ in some order.
The element $p_{0} \in A(E_{6}) \cap EVIII_{+}$ satisfies 
\[
\begin{split}
F^{+}(\mathrm{Ad}(p_{0})|_{\frak{e}_{6} \cap L}, \frak{e}_{6} \cap L) &\cong \mathbb{R} \oplus \frak{spin}(10), \\
F^{+}(\mathrm{Ad}(p_{0})|_{\frak{e}_{6} \cap V}, \frak{e}_{6} \cap V) &= \{ 0 \}.
\end{split}
\]
Hence, $\dim F^{+}(\mathrm{Ad}(p_{0}), \frak{e}_{6}) = 46$.
It follows that $\mathcal{O}_{E_{6}}(p_{0})$ is isometric to $EIII$.
Next, the element $q_{0} \in A(E_{6}) \cap EIX_{+}$ satisfies
\[
\begin{split}
F^{+}(\mathrm{Ad}(q_{0})|_{\frak{e}_{6} \cap L}, \frak{e}_{6} \cap L) &= \mathbb{R} \oplus \frak{spin}(6) \oplus \frak{spin}(4), \\
F^{+}(\mathrm{Ad}(q_{0})|_{\frak{e}_{6} \cap V}, \frak{e}_{6} \cap V) &= (\mathbb{H} \otimes^{+} \mathbb{C}) \oplus (\mathbb{H} \otimes^{-} \mathbb{C}).
\end{split}
\]
Thus, $\dim F^{+}(\mathrm{Ad}(q_{0}), \frak{e}_{6}) = 38$.
It follows that $\mathcal{O}_{E_{6}}(q_{0})$ is isometric to $EII$.
Consequently,
\[
EII_{+} = E_{6} \cap EIX_{+}, \qquad
EIII_{+} = E_{6} \cap EVIII_{+}.
\]
By Lemma \ref{E_{8}},
\[
\begin{split}
& A(EII_{+}) = A(E_{6}) \cap EIX_{+} = \{ \psi(\gamma_{k}^{0,0}, \gamma_{0}^{a,b}) \ ;\ 1 \leq k \leq 7, a,b = 0,1 \}, \\
& A(EIII_{+}) = A(E_{6}) \cap EVIII_{+} = \{ \psi(\gamma_{0}^{0,0}, \gamma_{0}^{a,b}) \ ;\ (a,b) \not= (0,0) \}.
\end{split}
\]
This proves the first assertion.
Finally, $A(E_{6}) = \{ r_{0} \} \sqcup A(EII_{+}) \sqcup A(EIII_{+})$, which completes the proof.
\end{proof}

In particular,
\[
EII_{+} = \mathcal{O}_{E_{6}}(q_{0}), \qquad
EIII_{+} = \mathcal{O}_{E_{6}}(p_{0}).
\]
As in Subsection 4.3, we divide the analysis of the remaining $E_{6}$-orbits into two parts.
We first consider the orbits through the points of $A(EVIII_{+}) \setminus A(E_{6})$, and then those through the points of $A(EIX_{+}) \setminus A(E_{6})$.


For $1 \leq i \leq 3$, the element $p_{i} \in A(EVIII_{+}) \setminus A(E_{6})$ satisfies
\[
\begin{split}
F^{+}(\mathrm{Ad}(p_{i})|_{\frak{e}_{6} \cap L}, \frak{e}_{6} \cap L) &= \mathbb{R} \oplus \frak{spin}(6) \oplus \frak{spin}(4), \\
F^{+}(\mathrm{Ad}(p_{i})|_{\frak{e}_{6} \cap V}, \frak{e}_{6} \cap V) &= (\mathbb{H} \otimes^{+} \mathbb{C}) \oplus (\mathbb{H} \otimes^{-} \mathbb{C}).
\end{split}
\]
Hence, $\dim F^{+}(\mathrm{Ad}(p_{i}), \frak{e}_{6}) = 38$.
It follows that $\mathcal{O}_{E_{6}}(p_{i})$ is isometric to $EII$.
For $1 \leq i \leq 3$, set
\[
EII_{i} = \mathcal{O}_{E_{6}}(p_{i}) \quad (1 \leq i \leq 3).
\]
For $4 \leq j \leq 7$, the element $p_{j} \in A(EVIII_{+}) \setminus A(E_{6})$ satisfies
\[
\begin{split}
F^{+}(\mathrm{Ad}(p_{j})|_{\frak{e}_{6} \cap L}, \frak{e}_{6} \cap L) &= \frak{spin}(5) \oplus \frak{spin}(5), \\
F^{+}(\mathrm{Ad}(p_{j})|_{\frak{e}_{6} \cap V}, \frak{e}_{6} \cap V) &= (\mathbb{H} \otimes^{+} \mathbb{R}e_{0}) \oplus (\mathbb{H}^{\perp} \otimes^{+} \mathbb{R}e_{1}) \oplus (\mathbb{H} \otimes^{-} \mathbb{R}e_{0}) \oplus (\mathbb{H}^{\perp} \otimes^{-} \mathbb{R}e_{1}).
\end{split}
\]
Therefore, $\dim F^{+}(\mathrm{Ad}(p_{j}), \frak{e}_{6}) = 36$.
Thus, $\mathcal{O}_{E_{6}}(p_{j})$ is isometric to $EI$.
For $4 \leq j \leq 7$, set
\[
EI_{j} = \mathcal{O}_{E_{6}}(p_{j}) \quad (4 \leq j \leq 7).
\]

\begin{lemm} \label{EII_{i}EI_{j}}
For $1 \leq i \leq 3$ and $4 \leq j \leq 7$, set $A(EII_{i}) = A(E_{8}) \cap EII_{i}$ and $A(EI_{j}) = A(E_{8}) \cap EI_{j}$.
Then,
\[
\begin{split}
& A(EII_{i}) = \big\{ \psi(\gamma_{k}^{0,0}, \gamma_{i}^{a,b}) \ ;\ 1 \leq k \leq 7, a,b = 0,1 \big\}, \\
& A(EI_{j}) = \big\{ \psi(\gamma_{k}^{0,0}, \gamma_{j}^{a,b}) \ ;\ 1 \leq k \leq 7, a,b = 0,1 \big\}.
\end{split}
\]
For every $p \in A(EVIII_{+}) \setminus A(E_{6})$, the orbit $\mathcal{O}_{E_{6}}(p)$ is one of $EII_{i} \ (1 \leq i \leq 3)$ and $EI_{j} \ (4 \leq j \leq 7)$.
\end{lemm}

\begin{proof}
Since $r_{1}$ and $r_{2}$ are the nontrivial elements of the centers of $E_{7}^{\alpha}$ and $E_{7}^{\beta}$ and $r_{3} = r_{1}r_{2}$, each element of $E_{6}$ commutes with $r_{i}\ (1 \leq i \leq 3)$.
Since $EII_{i} = r_{i} EII_{+}$ for any $1 \leq i \leq 3$, we have
\[
A(EII_{i}) = r_{i}A(EII_{+}) = \{ \psi(\gamma_{k}^{0,0}, \gamma_{i}^{a,b}) \ ;\ 1 \leq k \leq 7, a,b = 0,1 \}.
\]
We next determine $A(EI_{j})$.
For $1 \leq k \not= l \leq 7, \ 0 \leq m,n \leq 7$, and $a,b,c,d = 0,1$, we have
\[
\begin{split}
& \mathrm{Ad}( \psi(\gamma_{m}^{0,0}, \gamma_{k}^{a,b}) ) |_{\frak{su}^{\alpha, \beta}(3)} 
\not= \mathrm{Ad}( \psi(\gamma_{n}^{0,0}, \gamma_{l}^{c,d}) ) |_{\frak{su}^{\alpha, \beta}(3)}, \\
& \mathrm{Ad}( \psi(\gamma_{m}^{0,0}, \gamma_{k}^{a,b}) ) |_{\frak{su}^{\alpha, \beta}(3)} 
= \mathrm{Ad}( \psi(\gamma_{n}^{0,0}, \gamma_{k}^{c,d}) ) |_{\frak{su}^{\alpha, \beta}(3)}, \\
\end{split}
\]
Since $E_{6}$ is the identity component of $C(SU^{\alpha, \beta}(3), E_{8})$, the restriction of $\mathrm{Ad}(p)$ to $\frak{su}^{\alpha,\beta}(3)$ is constant along each $E_{6}$-orbit.
It follows that
\[
A(EI_{j}) \subset \{ \psi(\gamma_k^{0,0},\gamma_j^{a,b}) \ ;\ 1 \leq k \leq 7,\ a,b = 0,1 \}.
\]
Conversely, $\{ \gamma_{k}^{a,b} \ ;\ 1 \leq k \leq 7, a,b = 0,1 \} \subset \mathcal{O}_{Spin(8)}(\gamma_{1}^{0,0})$.
Since $\phi(Spin^{0}(8)) \subset E_{6}$, we obtain
\[
\{ \psi(\gamma_k^{0,0},\gamma_j^{a,b}) \ ;\ 1 \leq k \leq 7,\ a,b = 0,1 \} \subset A(EI_{j}).
\]
Therefore,
\[
A(EI_{j}) = \{ \psi(\gamma_{k}^{0,0}, \gamma_{j}^{a,b}) \ ;\ 1 \leq k \leq 7, a,b = 0,1 \}.
\]
Finally,
\[
A(EVIII_{+}) \setminus A(E_{6}) = \left( \bigsqcup_{1 \leq i \leq 3} A(EII_{i}) \right) \sqcup \left( \bigsqcup_{4 \leq j \leq 7} A(EI_{j}) \right),
\]
which completes the proof.
\end{proof}


For any $1 \leq i \leq 3$, the element $q_{i} \in A(EIX_{+}) \setminus A(E_{6})$ satisfies
\[
\begin{split}
F^{+}(\mathrm{Ad}(q_{i}), \frak{e}_{6} \cap L) &= \mathbb{R} \oplus \frak{spin}(10), \\
F^{+}(\mathrm{Ad}(q_{i}), \frak{e}_{6} \cap V) &= \{ 0 \}.
\end{split}
\]
Hence, $\dim F^{+}(\mathrm{Ad}(q_{i}), \frak{e}_{6}) = 46$ and the orbit $\mathcal{O}_{E_{6}}(q_{i})$ is isometric to $EIII$.
Set
\[
EIII_{i} = \mathcal{O}_{E_{6}}(q_{i}) \quad (1 \leq i \leq 3).
\]
For $1 \leq i \leq3$, we have $EIII_{i} = r_{i}EIII_{+}$.
For $4 \leq j \leq 7$, the element $q_{j} \in A(EIX_{+}) \setminus A(E_{6})$ satisfies
\[
\begin{split}
F^{+}(\mathrm{Ad}(q_{j}), \frak{e}_{6} \cap L) &= \frak{spin}(9), \\
F^{+}(\mathrm{Ad}(q_{j}), \frak{e}_{6} \cap V) &= (\mathbb{O} \otimes^{+} \mathbb{R}e_{1}) \oplus (\mathbb{O} \otimes^{-} \mathbb{R}e_{1}).
\end{split}
\]
Hence, $\dim F^{+}(\mathrm{Ad}(q_{j}), \frak{e}_{6}) = 52$ and the orbit $\mathcal{O}_{E_{6}}(q_{j})$ is isometric to $EIV$.
Set
\[
EIV_{j} = \mathcal{O}_{E_{6}}(q_{j}) \quad (4 \leq j \leq 7).
\]
Since each element of $E_{6}^{\alpha, \beta}$ commutes with $r_{i} \ (1 \leq i \leq 3)$, we have
\[
\mathcal{O}_{E_{6}}(r_{i}) = \{ r_{i} \} \quad (1 \leq i \leq 3).
\]

\begin{lemm} \label{EIII_{i}EIV_{j}}
For $1 \leq i \leq 3$ and $4 \leq j \leq 7$, set $A(EIII_{i}) = A(E_{8}) \cap EIII_{i}$ and $A(EIV_{j}) = A(E_{8}) \cap EIV_{j}$.
Then,
\[
\begin{split}
& A(EIII_{i}) = \big\{ \psi(\gamma_{0}^{0,0}, \gamma_{i}^{a,b}) \ ;\ (a,b) \not= (0,0) \big\}, \\
& A(EIV_{j}) = \big\{ \psi(\gamma_{0}^{0,0}, \gamma_{j}^{a,b}) \ ;\ a,b = 0,1 \big\}.
\end{split}
\]
For every $p \in A(EVIII_{+}) \setminus A(E_{6})$, the orbit $\mathcal{O}_{E_{6}}(p)$ is one of $\{ r_{i} \}, EIII_{i} \ (1 \leq i \leq 3)$ and $EIV_{j} \ (4 \leq j \leq 7)$. 
\end{lemm}

\begin{proof}
Since $EIII_{i} = r_{i}EIII_{+}$ for $1 \leq i \leq 3$, we obtain
\[
A(EIII_{i}) = r_{i}A(EIII_{+}) = \{ \psi(\gamma_{0}^{0,0}, \gamma_{i}^{a,b}) \ ;\ (a,b) \not= (0,0) \}.
\]
As in the proof of Lemma \ref{EII_{i}EI_{j}}, it follows that $A(EIV_{j})$ is contained in the right-hand side.
On the other hand, the subgroup $H = \mathrm{exp}(\psi(0, \mathbb{R}t_{8}(1,0,0,0) \oplus \mathbb{R}t_{8}(0,1,1,1)))$ of $E_{6}$ satisfies $\{ \psi(\gamma_{0}^{0,0}, \gamma_{j}^{a,b}) \ ;\ a,b = 0,1 \} \subset \mathcal{O}_{H}(q_{j})$.
Hence, $A(EIV_{j})$ contains the right-hand side, and
\[
A(EIV_{j}) = \{ \psi(\gamma_{0}^{0,0}, \gamma_{j}^{a,b}) \ ;\ a,b = 0,1 \}.
\]
Finally
\[
A(EIX_{+}) \setminus A(E_{6}) = \{ r_{1}, r_{2}, r_{3} \} \sqcup \bigsqcup_{1 \leq i \leq 3}A(EIII_{i}) \sqcup \bigsqcup_{4 \leq j \leq 7}A(EIV_{j}),
\]
which completes the proof.
\end{proof}

By Lemma \ref{EII_{+}EIII_{+}}, \ref{EIII_{i}EIV_{j}}, \ref{EII_{i}EI_{j}}, we obtain the following.

\begin{lemm} \label{E_{6}}
Each orbit $\mathcal{O}_{E_{6}}(p) \ (p \in A(E_{8}))$ is one of
\[
\begin{array}{cccccccccc}
\{ r_{0} \} & EIII_{+} = \mathcal{O}_{E_{6}}(p_{0}), 
& EII_{i} = \mathcal{O}_{E_{6}}(p_{i}), 
& EI_{j} = \mathcal{O}_{E_{6}}(p_{j}), \\
\{ r_{i} \} & EII_{+} = \mathcal{O}_{E_{6}}(q_{0}), 
& EIII_{i} = \mathcal{O}_{E_{6}}(q_{i}), 
& EIV_{j} = \mathcal{O}_{E_{6}}(q_{j}),
\end{array}
\]
where $1 \leq i \leq 3$ and $4 \leq j \leq 7$.
Furthermore,
\[
\begin{array}{llll}
EIII_{+} \cup EII_{1} \subset EVI_{+}, & EII_{+} \cup EIII_{1} \subset EVI'_{+}, \\
EII_{2} \cup EII_{3} \subset EV_{1}, & EI_{4} \cup EI_{5} \subset EV_{2}, & EI_{6} \cup EI_{7} \subset EV_{3}, \\
EIII_{2} \cup EIII_{3} \cup \{ r_{2}, r_{3} \} \subset EVII_{1}, & EIV_{4} \cup EIV_{5} \subset EVII_{2}, & EIV_{6} \cup EIV_{7} \subset EVII_{3}.
\end{array}
\]
\end{lemm}


\subsection{The $F_{4}$-orbits} \label{F_{4}-orbit}

As in the previous subsections, we first determine the intersection $A(F_{4}) = A(E_{8}) \cap F_{4}$.
The subgroup
\[
A(E_{8}) \cap C(G'_{2}, E_{8}) = \big\{ \psi(\gamma_{k}^{0,0}, \gamma_{0}^{a,b}) \ ;\ 1 \leq k \leq 7, a,b = 0,1 \big\}
\]
is contained in the subgroup $\phi(Spin^{0}(8)) \cong Spin(8)$ of $F_{4}$.
Therefore, we have
\[
A(F_{4}) = A(E_{8}) \cap C(G'_{2}, E_{8}).
\]
Note that $A(F_{4}) = A(E_{6})$.

According to the classification of polars \cite{Chen-Nagano1}, $F_{4}$ has two polars, isometric to $FI$ and $FII$, respectively.
Let $FI_{+}$ and $FII_{+}$ denote the polars of $e$ isometric to $FI$ and $FII$, respectively.

\begin{lemm} \label{FI_{+}FII_{+}}
Set $A(FI_{+}) = A(E_{8}) \cap FI_{+}$ and $A(FII_{+}) = A(E_{8}) \cap FII_{+}$.
Then,
\[
\begin{split}
& A(FI_{+}) = \{ \psi(\gamma_{k}^{0,0}, \gamma_{0}^{a,b}) \ ;\ 1 \leq k \leq 7, a,b = 0,1 \}, \\
& A(FII_{+}) = \{ \psi(\gamma_{0}^{0,0}, \gamma_{0}^{a,b}) \ ;\ (a,b) \not= (0,0) \}. 
\end{split}
\]
For every $p \in A(F_{4})$, the orbit $\mathcal{O}_{F_{4}}(p)$ is one of $\{ r_{0} \}, FI_{+}$, and $FII_{+}$.
\end{lemm}

\begin{proof}
The two intersections $F_{4} \cap EVIII_{+}$ or $F_{4} \cap EIX_{+}$ are the polars $FI_{+}$ and $FII_{+}$ in some order.
The element $p_{0} \in A(F_{4}) \cap EVIII_{+}$ satisfies 
\[
\begin{split}
& F^{+}(\mathrm{Ad}(p_{0}), \frak{f}_{4} \cap L) = \frak{spin}(9), \\
& F^{+}(\mathrm{Ad}(p_{0}), \frak{f}_{4} \cap V) = \{ 0 \}. \\
\end{split}
\]
It follows that $\dim F^{+}(\mathrm{Ad}(p_{0}), \frak{f}_{4}) = 36$.
Hence, the orbit $\mathcal{O}_{F_{4}}(p_{0})$ is isometric to $FII$.
Furthermore, the element $q_{0} \in A(F_{4}) \cap EIX_{+}$ satisfies
\[
\begin{split}
& F^{+}(\mathrm{Ad}(q_{0}), \frak{f}_{4} \cap L) = \frak{spin}(5) \oplus \frak{spin}(4), \\
& F^{+}(\mathrm{Ad}(q_{0}), \frak{f}_{4} \cap V) = (\mathbb{H} \otimes^{+} \mathbb{R}) \oplus ( \mathbb{H} \otimes^{-} \mathbb{R}). \\
\end{split}
\]
It follows that $\dim F^{+}(\mathrm{Ad}(q_{0}), \frak{f}_{4}) = 24$.
Hence, the orbit $\mathcal{O}_{F_{4}}(q_{0})$ is isometric to $FI$.
Thus,
\[
FI_{+} = F_{4} \cap EIX_{+}, \qquad
FII_{+} = F_{4} \cap EVIII_{+}.
\]
By Lemma \ref{E_{8}}, we have
\[
\begin{split}
& A(FI_{+}) = \{ \psi(\gamma_{k}^{0,0}, \gamma_{0}^{a,b}) \ ;\ 1 \leq k \leq 7, a,b = 0,1 \}, \\
& A(FII_{+}) = \{ \psi(\gamma_{0}^{0,0}, \gamma_{0}^{a,b}) \ ;\ (a,b) \not= (0,0) \}. 
\end{split}
\]
This proves the first assertion.
Finally, $A(F_{4}) = \{ r_{0} \} \sqcup A(FI_{+}) \sqcup A(FII_{+})$, which completes the proof.
\end{proof}

Note that
\[
FI_{+} = \mathcal{O}_{F_{4}}(q_{0}), \quad FII_{+} = \mathcal{O}_{F_{4}}(p_{0}).
\]
As in the previous subsections, we divide the arguments into two parts.


For any $1 \leq k \leq 7$, the element $p_{k} \in A(EVIII_{+}) \setminus A(F_{4})$ satisfies
\[
\begin{split}
& F^{+}(\mathrm{Ad}(p_{k}), \frak{f}_{4} \cap L) = \frak{spin}(4) \oplus \frak{spin}(5), \\
& F^{+}(\mathrm{Ad}(p_{k}), \frak{f}_{4} \cap V) = (\mathbb{H} \otimes^{+} \mathbb{R}) \oplus (\mathbb{H} \otimes^{-} \mathbb{R}).
\end{split}
\]
Hence, $\dim F^{+}(\mathrm{Ad}(p_{k}), \frak{f}_{4}) = 24$.
It follows that the orbit $\mathcal{O}_{F_{4}}(p_{k})$ is isometric to $FI$.
For $1 \leq k \leq 7$, set
\[
FI_{k} = \mathcal{O}_{F_{4}}(p_{k}).
\]

\begin{lemm} \label{FI_{k}}
For any $1 \leq k \leq7$, set $A(FI_{k}) = A(E_{8}) \cap FI_{k}$.
Then,
\[
A(FI_{k}) = \big\{ \psi(\gamma_{l}^{0,0}, \gamma_{k}^{a,b}) \ ;\ 1 \leq l \leq 7, a,b = 0,1 \big\}.
\]
For every $p \in A(EVIII_{+}) - A(F_{4})$, the orbit $\mathcal{O}_{F_{4}}(p)$ is one of $FI_{k} \ (1 \leq k \leq 7)$.
\end{lemm}

\begin{proof}
Since $r_{k} \in G_{2}' \subset \phi(Spin^{1}(8))$ for $1 \leq k \leq 7$, each element of $F_{4}$ commutes with $r_{k}$.
We have $FI_{k} = r_{k}FI_{+}$ for $1 \leq k \leq 7$.
Consequently,
\[
A(FI_{k}) = r_{k}A(FI_{+}) = \{ \psi(\gamma_{l}^{0,0}, \gamma_{k}^{a,b}) \ ;\ 1 \leq l \leq 7, a,b = 0,1 \}.
\]
Finally,
\[
A(EVIII_{+}) \setminus A(F_{4}) = \bigsqcup_{1 \leq k \leq 7}A(FI_{k}),
\]
which completes the proof.
\end{proof}


For any $1 \leq k \leq 7$, the element $q_{k} \in A(EIX_{+}) \setminus A(F_{4})$ satisfies
\[
\begin{split}
& F^{+}(\mathrm{Ad}(q_{k}), \frak{f}_{4} \cap L) = \frak{spin}(9), \\
& F^{+}(\mathrm{Ad}(q_{k}), \frak{f}_{4} \cap V) = \{ 0 \}.
\end{split}
\]
Hence, $\dim F^{+}(\mathrm{Ad}(q_{k}), \frak{f}_{4}) = 36$.
It follows that $\mathcal{O}_{F_{4}}(q_{k})$ is isometric to $FII$.
Set
\[
FII_{k} = \mathcal{O}_{F_{4}}(q_{k}) \quad (1 \leq k \leq 7).
\]
Since each element of $F_{4}$ commutes with $r_{k}$ for any $1 \leq k \leq 7$,
\[
\mathcal{O}_{F_{4}}(r_{k}) = \{ r_{k} \}.
\]

\begin{lemm} \label{FII_{k}}
For any $1 \leq k \leq7$, set $A(FII_{k}) = A(E_{8}) \cap FII_{k}$.
Then,
\[
A(FII_{k}) = \big\{ \psi(\gamma_{0}^{0,0}, \gamma_{k}^{a,b}) \ ;\ (a,b) \not= (0,0) \big\}.
\]
For every $p \in A(EIX_{+}) \setminus A(F_{4})$, the orbit $\mathcal{O}_{F_{4}}(p)$ is one of $\{ r_{k} \}$ and $FII_{k} \ (1 \leq k \leq 7)$.
\end{lemm}

\begin{proof}
Since $FII_{k} = r_{k}FII_{+} \ (1 \leq k \leq 7)$, we have
\[
A(FII_{k}) = r_{k}A(FII_{+}) = \{ \psi(\gamma_{0}^{0,0}, \gamma_{k}^{a,b}) \ ;\ (a,b) \not= (0,0) \}.
\]
Finally,
\[
A(EIX_{+}) \setminus A(F_{4}) = \{ r_{k} \ ;\ 1 \leq k \leq 7 \} \sqcup  \bigsqcup_{1 \leq k \leq 7}A(FII_{k}),
\]
which completes the proof.
\end{proof}

By Lemmas \ref{FI_{+}FII_{+}}, \ref{FI_{k}}, and \ref{FII_{k}}, we obtain the following.

\begin{lemm} \label{F_{4}}
For every $p \in A(E_{8})$, the orbit $\mathcal{O}_{F_{4}}(p)$ is one of 
\[
\{ r_{0} \}, \ \{ r_{k} \}, \ FI_{+}, \ FI_{k}, \ FII_{+}, \ FII_{k} \ (1 \leq k \leq 7).
\]
Furthermore,
\[
\begin{array}{llcccc}
FI_{+} \subset EII_{+}, & FII_{+} \subset EIII_{+}, \\
FI_{i} \subset EII_{i}, & FII_{i} \subset EIII_{i} \quad (1 \leq i \leq 3), \\
FI_{j} \subset EI_{j}, & FII_{j} \cup \{ r_{j} \} \subset EIV_{j} \quad (4 \leq j \leq 7).
\end{array}
\]
\end{lemm}


\subsection{The $G_{2}$-orbits} \label{G_{2}-orbit}

An orbit $\mathcal{O}_{G_{2}}(p) \ (p \in A(E_{8}))$ has positive dimension if and only if it is isometric to $G_{2}/SO(4)$.
For any $0 \leq l \leq 7$ and $a,b = 0,1$, set
\[
(G_{2}/SO(4))_{l}^{a,b} = \mathcal{O}_{G_{2}}(\psi(\gamma_{1}^{0,0}, \gamma_{l}^{a,b})).
\]
On the other hand,
\[
\mathcal{O}_{G_{2}}(\psi(\gamma_{0}^{0,0}, \gamma_{l}^{a,b})) = \{ \psi(\gamma_{0}^{0,0}, \gamma_{l}^{a,b}) \}.
\]

\begin{lemm} \label{G_{2}}
Set $A((G_{2}/SO(4))_{l}^{a,b}) = A(E_{8}) \cap (G_{2}/SO(4))_{l}^{a,b}$ for $0 \leq l \leq 7$ and $a,b = 0,1$.
Then,
\[
A((G_{2}/SO(4))_{l}^{a,b}) = \{ \psi(\gamma_{k}^{0,0}, \gamma_{l}^{a,b}) \ ;\ 1 \leq k \leq 7 \}.
\]
For every $p \in A(E_{8})$, the orbit $\mathcal{O}_{G_{2}}(p)$ is one of $\{ \psi(\gamma_{0}^{0,0}, \gamma_{l}^{a,b}) \}$ and $(G_{2}/SO(4))_{l}^{a,b}$ $\ (0 \leq l \leq 7,$ $ a,b = 0,1)$.
\end{lemm}

\begin{proof}
By the classification of polars \cite{Chen-Nagano1}, $G_{2}$ has a unique polar, which is isometric to $G_{2}/SO(4)$.
Since $A(G_{2}) = \{ \gamma_{k} \ ; \ 1 \leq k \leq 7 \}$ is a maximal antipodal subgroup of $G_{2}$, 
\[
\{ \gamma_{k} \ ;\ 1 \leq k \leq 7 \} \subset \mathcal{O}_{G_{2}}(\gamma_{1}) \cong G_{2}/SO(4).
\]
Moreover, in $Spin(8)$, for every $h \in G_{2}$ and $(a,b) \not= (0,0)$, there does not exist $g \in G_{2}$ such that $g^{0,0} h^{0,0} (g^{0,0})^{-1} = h^{a,b}$.
Hence, we have
\[
A((G_{2}/SO(4))_{l}^{a,b}) = \{ \psi(\gamma_{k}^{0,0}, \gamma_{l}^{a,b}) \ ;\ 1 \leq k \leq 7 \}.
\]
Finally,
\[
A(E_{8}) = \{ \psi(\gamma_{0}^{0,0}, \gamma_{l}^{a,b}) \ ;\ 0 \leq l \leq 7, a,b = 0, 1 \} \sqcup \bigsqcup_{0 \leq l \leq 7, a,b = 0,1}A((G_{2}/SO(4))_{l}^{a,b}),
\]
which completes the proof.
\end{proof}

Then,
\[
(G_{2}/SO(4))_{0}^{a,b} \subset FI_{+}, \ 
(G_{2}/SO(4))_{k}^{a,b} \subset FI_{k} \quad (1 \leq k \leq 7, a,b = 0,1).
\]
Furthermore,
\[
\psi(\gamma_{0}^{0,0}, \gamma_{0}^{a,b}) \in FII_{+}, \ 
\psi(\gamma_{0}^{0,0}, \gamma_{k}^{a,b}) \in FI_{k} \quad (1 \leq k \leq 7, (a,b) \not= (0,0)).
\]


\subsection{The classification of orbits}

Combining Lemmas \ref{E_{8}}, \ref{E_{7}}, \ref{E_{6}}, \ref{F_{4}}, and \ref{G_{2}}, we obtain the following classification.

\begin{thm} \label{main-2}
For each $G \in \{ G_{2}, F_{4}, E_{6}, E_{7}, E_{8} \}$, all possible orbits $\mathcal{O}_{G}(p)$ with $p \in A(E_{8})$ are listed in Table \ref{classification}.
All inclusion relations among the positive-dimensional orbits except $G_{2}/SO(4)$ are given by the following diagrams.
An arrow indicates inclusion of the orbit at its tail into the orbit at its head.
\[
\small
\xymatrix@C=5pt@R=11pt
{
 & & FI_{1} \ar[d] & & FI_{2} \ar[d] & & \\
 & & EII_{1} \ar[d] & & EII_{2} \ar[d] & & \\
FII_{+} \ar[r] & EIII_{+} \ar[r] & EVI_{+} \ar[dr] & & EV_{1} \ar[dl] & EII_{3} \ar[l] & FI_{3} \ar[l] \\
 & & & EVIII_{+} & & & \\
FI_{7} \ar[r] & EI_{7} \ar[r] & EV_{3} \ar[ur] & & EV_{2} \ar[ul] & EI_{4} \ar[l] & FI_{4} \ar[l] \\
 & & EI_{6} \ar[u] & & EI_{5} \ar[u] & & \\
 & & FI_{6} \ar[u] & & FI_{5} \ar[u] & & \\
}
\xymatrix@C=5pt@R=11pt
{
 & & FII_{1} \ar[d] & & FII_{2} \ar[d] & & \\
 & & EIII_{1} \ar[d] & & EIII_{2} \ar[d] & & \\
FI_{+} \ar[r] & EII_{+} \ar[r] & EVI'_{+} \ar[dr] & & EVII_{1} \ar[dl] & EIII_{3} \ar[l] & FII_{3} \ar[l] \\
 & & & EIX_{+} & & & \\
FII_{7} \ar[r] & EIV_{7} \ar[r] & EVII_{3} \ar[ur] & & EVII_{2} \ar[ul] & EIV_{4} \ar[l] & FII_{4} \ar[l] \\
 & & EIV_{6} \ar[u] & & EIV_{5} \ar[u] & & \\
 & & FII_{6} \ar[u] & & FII_{5}. \ar[u] & & \\
}
\]
Furthermore, for $1\leq k\leq7$ and $a,b = 0,1$,
\[
(G_{2}/SO(4))_{0}^{a,b}\subset FI_{+},
\qquad
(G_{2}/SO(4))_{k}^{a,b}\subset FI_{k}
\]

\begin{table}[h]
\caption{The classification of orbits} \label{classification}
\centering
\begin{tabular}{l|l|lll} \hline
$G$ & \text{zero-dimensional orbits} & \text{positive-dimensional orbits} \\ \hline
$E_{8}$ & $\{ r_{0} \}$ & 
$\begin{array}{llll} 
EVIII_{+} = \mathcal{O}_{E_{8}}(p_{0}), \\ 
EIX_{+} = \mathcal{O}_{E_{8}}(q_{0})
\end{array}$ \\ \hline
$E_{7}$ & $\{ r_{0} \}, \{ r_{1} \}$ & 
$\begin{array}{llllllll}
EVI_{+} = \mathcal{O}_{E_{7}}(p_{0}), & EV_{i} = \mathcal{O}_{E_{7}}(p_{2i}), \\ 
EVI_{+}' = \mathcal{O}_{E_{7}}(q_{0}), & EVII_{i} = \mathcal{O}_{E_{7}}(q_{2i}) \\
\end{array}$ \quad $(1 \leq i \leq 3)$ \\ \hline
$E_{6}$ & 
$\{ r_{0} \}, \{ r_{1} \}, \{ r_{2} \}, \{ r_{3} \}$ &
$\begin{array}{lllllllll}
EIII_{+} = \mathcal{O}_{E_{6}}(p_{0}), & EII_{i} = \mathcal{O}_{E_{6}}(p_{i}), & EI_{j} = \mathcal{O}_{E_{6}}(p_{j}), & \\
EII_{+} = \mathcal{O}_{E_{6}}(q_{0}), & EIII_{i} = \mathcal{O}_{E_{6}}(q_{i}), & EIV_{j} = \mathcal{O}_{E_{6}}(q_{j}) & \\
\end{array}$
$\Big( \ \begin{matrix} 1 \leq i \leq 3 \\ 4 \leq j \leq 7 \end{matrix} \ \Big)$ \\ \hline
$F_{4}$ &
$\begin{matrix}
\{ r_{0} \}, \{ r_{1} \}, \{ r_{2} \}, \{ r_{3} \}, \\
\{ r_{4} \}, \{ r_{5} \}, \{ r_{6} \}, \{ r_{7} \}
\end{matrix}$ &
$\begin{array}{lllllllllll}
FII_{+} = \mathcal{O}_{F_{4}}(p_{0}), & FI_{k} = \mathcal{O}_{F_{4}}(p_{k}), \\
FI_{+} = \mathcal{O}_{F_{4}}(q_{0}), & FII_{k} = \mathcal{O}_{F_{4}}(q_{k}) \\
\end{array}$
\quad $(1 \leq k \leq 7)$ \\ \hline
$G_{2}$ & 
$\begin{matrix}
\{ \psi(\gamma_{0}^{0,0}, \gamma_{l}^{a,b}) \}, \\
(0 \leq l \leq 7, a,b = 0,1)
\end{matrix}$ &
$(G_{2}/SO(4))_{l}^{a,b} = \mathcal{O}_{G_{2}}(\psi(\gamma_{1}^{0,0}, \gamma_{l}^{a,b})) \quad  (0 \leq l \leq 7, a,b = 0,1)$ \\ \hline
\end{tabular}
\end{table}

\end{thm}

\begin{remark}

For $\frak{g} = \frak{f}_{4}, \frak{e}_{6}, \frak{e}_{7}, \frak{e}_{8}$, define an equivalence relation $\sim_{\frak{g}}$ on $A(E_{8})$ as follows.
For $p = \psi(a,b)$ and $q = \psi(c,d) \in A(E_{8})$, set $p \sim_{\frak{g}} q$ if and only if $p$ and $q$ satisfy the follwing conditions:
\[
\begin{split}
& \text{(1)}\ 
\dim \mathrm{Ad}(p)|_{\frak{g}} = \dim \mathrm{Ad}(q)|_{\frak{g}} \ \text{and}\ 
\dim \mathrm{Ad}(p)|_{\frak{g}^{\perp}} = \dim \mathrm{Ad}(q)|_{\frak{g}^{\perp}}, \\
& \text{(2)}\ 
\pi_{V}(\Delta_{8}^{+}(b)|_{V}) = \pi_{V}(\Delta_{8}^{+}(d)|_{V}),
\end{split}
\]
where $V$ is given by
\[
V = 
\begin{cases} 
\mathbb{O}^{\perp} = \{ 0 \} & (\frak{g} = \frak{e}_{8}), \\
\mathbb{H}^{\perp} & (\frak{g} = \frak{e}_{7}), \\
\mathbb{C}^{\perp} & (\frak{g} = \frak{e}_{6}), \\
\mathbb{R}^{\perp} & (\frak{g} = \frak{f}_{4}), \\
\end{cases}
\]
and $\mathbb{O}^{\perp}, \mathbb{H}^{\perp}, \mathbb{C}^{\perp},$ and $\mathbb{R}^{\perp}$ are the orthogonal complements in $\mathbb{O}$, and $\pi_{V}$ is the natural projection $\pi_{V} : GL(V) \to PGL(V)$.
Note that $\pi_{V}(\Delta^{+}(b)|_{V}) = \pi_{V}(\Delta^{+}(d)|_{V})$ if and only if $\pi_{V}(\Delta^{-}(b)|_{V}) = \pi_{V}(\Delta^{-}(d)|_{V})$ in this case.
Then, each equivalence class corresponds to a $G$-orbit one-to-one.
Thus, the $G$-orbits are characterized by the adjoint representation and the structure arising from the chain
\[
\mathbb{R} \subset \mathbb{C} \subset \mathbb{H} \subset \mathbb{O}.
\]
We also observe a numerical pattern related to this chain. 
As the group becomes larger along $F_{4} \subset E_{6} \subset E_{7} \subset E_{8}$, the numbers of positive dimensional orbits of the smaller group contained in an orbit of the larger group are 1,2, and 4. 
These numbers correspond to the dimensions of the relevant orthogonal complements arising from the chain $\mathbb{R} \subset \mathbb{C} \subset \mathbb{H} \subset \mathbb{O}$.
Furthermore, the numbers of fixed points also correspond to the dimensions of each of $\mathbb{R,C,H,O}$.
However, we do not have a conceptual explanation for this phenomenon.

\end{remark}

Since all the orbits are totally geodesic submanifolds of $E_{8}$, each inclusion displayed in Theorem \ref{main-2} is totally geodesic.
Combining Theorem \ref{main-2} with some arguments about the nonexistence, we obtain the following.

\begin{coro} \label{main-3}
Let $M$ and $N$ be distinct symmetric spaces appearing in the following diagram.
\[
\small
\xymatrix@C=20pt@R=15pt
{
 & & & E_{8} & & & \\
 & & EVIII \ar[ur] & & EIX \ar[ul] & &  \\
 & EV \ar[ur] & & EVI \ar[ul] \ar[ur] & & EVII \ar[ul] & \\
 EI  \ar[ur] & & EII \ar[ul] \ar[ur] & & EIII \ar[ul] \ar[ur] & & EIV \ar[ul] \\
 & FI \ar[ul] \ar[ur] & & & & FII. \ar[ul] \ar[ur] & \\
}
\]
Then, there exists a totally geodesic embedding $M \to N$ if and only if there is a directed path from $M$ to $N$ in the diagram.
\end{coro}

\begin{proof}
If there exists a directed path from $M$ to $N$ in the diagram, then the corresponding chain of inclusions in Theorem \ref{main-2} yields a totally geodesic embedding $M\hookrightarrow N$.
Therefore, it suffices to prove the nonexistence of totally geodesic embeddings for pairs that are not connected by a directed path.
First, we show that there exist no totally geodesic embeddings 
\[
\begin{array}{llll}
FI \to EIII, & FI \to EIV, & FI \to EVII, & EII \to EVII, \\
EI \to EVI, & EI \to EVII, & EI \to EIX, & EV \to EIX.
\end{array}
\]
It suffices to consider the case of $FI \to EVII$ and $EI \to EIX$.
Since
\[
\mathrm{rank}\,FI = 4 > 3 = \mathrm{rank}\,EVII, \quad
\mathrm{ranK}\,EI = 6 > 4 = \mathrm{rank}\,EIX,
\]
no such embeddings exist.
Next, we show that there exist no totally geodesic embeddings
\[
\begin{array}{llll}
FII \to EI, & FII \to EII, & FII \to EV, & EIII \to EV, \\
EIV \to EV, & EIV \to EVI, & EVI \to EVIII, & EVII \to EVIII.
\end{array}
\]
It suffices to consider the case of $FII \to EV$ and $EIV \to EVIII$.
Suppose, for contradiction, that there exists a totally geodesic embedding $FII \to EV$.
Since $FII = F_{4}/Spin(9)$ and $EV = E_{7}/(SU(8)/\mathbb{Z}_{2})$, we have $Spin(9) \subset SU(8)/\mathbb{Z}_{2}$.
In particular, $Spin(9) \subset SU(8)$ since $Spin(9)$ is simply connected.
Hence, we have a complex representation of $Spin(9)$ on $\mathbb{C}^{8}$.
Since the smallest nontrivial irreducible complex representation of of $Spin(9)$ has dimension $9$, this representation must be trivial.
On the other hand, every maximal torus of $Spin(9)$ is contained in a maximal torus of $SU(8)$, so some element of a maximal torus of $Spin(9)$ acts nontrivially on $\mathbb{C}^{8}$.
This contradiction shows that no such embeddings $FII \to EV$ exist.
For the case of $EIV \to EVIII$, suppose that there exists such an embedding.
Then, $F_{4} \subset Spin(16)$ and we have a representation of $F_{4}$ on $\mathbb{R}^{16}$.
Since the smallest nontrivial irreducible real representation of $F_{4}$ has dimension $26$, we again obtain a contradiction.
Therefore, no totally geodesic embeddings $EIV \to EVIII$ exist.
\end{proof}

For the symmetric space $G_{2}/SO(4)$, there exists a totally geodesic embedding $G_{2}/SO(4) \to FI$.
However, there are no totally geodesic embeddings such that $G_{2}/SO(4) \to FII$ since $\mathrm{rank}(G_{2}/SO(4)) = 2 > 1 = \mathrm{rank}FII$.


\subsection{Geometric features of the chains}

We observe that several chains in the diagram of Cororally \ref{main-3} have distinctive geometric features.
The symmetric spaces in the chain
\[
FI \rightarrow EII \rightarrow EVI \rightarrow EIX
\]
are all Wolf spaces, namely quaternionic K\"{a}hler symmetric spaces, and each inclusion is a totally geodesic quaternionic embedding.
Recall that any quaternionic submanifold of a quaternionic K\"{a}hler submanifold is totally geodesic \cite{Alekseevsky}. 
Likewise, the symmetric spaces in the chain
\[
EIII \rightarrow EVII
\]
are all Hermitian symmetric spaces, and the inclusion is a totally geodesic complex embedding.
The embeddings
\[
EIII \rightarrow EVI, \quad
EVII \rightarrow EIX
\]
are totally geodesic totally complex embeddings, which form a special class of embeddings of a complex manifold into a quaternionic manifold.
Totally complex embeddings were introduced by Funabashi as an analogue of totally real embeddings \cite{Funabashi}.
On the other hand, the embeddings
\[
FII \rightarrow EIII, \quad
EIV \rightarrow EVII
\]
are totally real.
In particular, $FII$ is a Lagrangian submanifold of $EIII$, whereas $EIV$ is not Lagrangian in $EVII$ since $\dim EIV = 26 < 27 = (1/2)\dim EVII$.

The chain
\[
FII \to EIII \to EVI \to EVIII,
\]
consists of the Rosenfeld planes.
The Rosenfeld planes were introduced by Rosenfeld as a generalization of projective planes \cite{Rosenfeld} and are represented by
\[
\begin{array}{ll}
FII = \mathbb{O}P^{2} = (\mathbb{O} \otimes \mathbb{R})P^{2}, &
EIII = (\mathbb{O} \otimes \mathbb{C})P^{2}, \\
EVI = (\mathbb{O} \otimes \mathbb{H})P^{2}, &
EVIII = (\mathbb{O} \otimes \mathbb{O})P^{2}.
\end{array}
\]

The two chains
\[
FI \to EI \to EV \to EVIII, \quad FII \to EIV \to EVII \to EIX
\]
exhibit another common feature, which is visible from their Satake diagrams.
Let $\frak{g}$ be the Lie algebra of the isometry group and $\frak{g} = \frak{k} \oplus \frak{m}$ the standard decomposition associated with the involution, where $\frak{k}$ and $\frak{m}$ are the $\pm1$-eigenspaces, respectively.
Their Satake diagrams are as follows:
\[
\begin{array}{lclccccc}
FI &
\xymatrix@C=10pt@R=10pt
{
\circ \ar@{-}[r] & \circ \ar@{<=}[r] & \circ \ar@{-}[r] & \circ
} & 
FII &
\xymatrix@C=10pt@R=10pt
{
\circ \ar@{-}[r] & \bullet \ar@{<=}[r] & \bullet \ar@{-}[r] & \bullet
} \\ 
 \\
EI &
\xymatrix@C=10pt@R=10pt
{
\circ \ar@{-}[r] & \circ \ar@{-}[r] & \circ \ar@{-}[r] \ar@{-}[d] & \circ \ar@{-}[r] & \circ \\
 & & \circ & & \\
} &
EIV &
\xymatrix@C=10pt@R=10pt
{
\circ \ar@{-}[r] & \bullet \ar@{-}[r] & \bullet \ar@{-}[r] \ar@{-}[d] & \bullet \ar@{-}[r] & \circ \\
 & & \bullet & & \\
} \\
EV & 
\xymatrix@C=10pt@R=10pt
{
\circ \ar@{-}[r] & \circ \ar@{-}[r] & \circ \ar@{-}[r]  & \circ \ar@{-}[r] \ar@{-}[d] & \circ \ar@{-}[r] & \circ \\
 & & & \circ & & \\
} &
EVII & 
\xymatrix@C=10pt@R=10pt
{
\circ \ar@{-}[r] & \circ \ar@{-}[r] & \bullet \ar@{-}[r] & \bullet \ar@{-}[r] \ar@{-}[d] & \bullet \ar@{-}[r] & \circ \\
 & & & \bullet & & \\
} \\
EVIII & 
\xymatrix@C=10pt@R=10pt
{
\circ \ar@{-}[r] & \circ \ar@{-}[r] & \circ \ar@{-}[r] & \circ \ar@{-}[r]  & \circ \ar@{-}[r] \ar@{-}[d] & \circ \ar@{-}[r] & \circ \\
 & & & & \circ & & \\
} &
EIX & 
\xymatrix@C=10pt@R=10pt
{
\circ \ar@{-}[r] & \circ \ar@{-}[r] & \circ \ar@{-}[r] & \bullet \ar@{-}[r] & \bullet \ar@{-}[r] \ar@{-}[d] & \bullet \ar@{-}[r] & \circ \\
 & & & & \bullet & & \\
} \\
\end{array}
\]
For the first chain $FI \to EI \to EV \to EVIII$, the Satake diagrams coincide with the base Dynkin diagrams.
In contrast, for the second chain $FII \to EIV \to EVII \to EIX$, the black vertices form the Dynkin diagrams of type $B_{3}$ or $D_{4}$, corresponding to $\frak{spin}(7)$ or $\frak{spin}(8)$.
Hence, for any maximal abelian subalgebra $\frak{a}$ of $\frak{m}$, 
\[
C(\frak{a}, \frak{k}) = \{ 0 \}
\]
along the first chain, whereas
\[
C(\frak{a}, \frak{k}) = \begin{cases} \frak{spin}(7) & (FII), \\ \frak{spin}(8) & (EIV, EVII, EIX) \end{cases}
\]
along the second chain.
Within each row of the inclusion diagram, the spaces in the first chain have the smallest possible value of $\dim C(\frak{a}, \frak{k})$, while those in the second chain have the largest.

These observations indicate that the chains appearing in Corollary \ref{main-3} are not arbitrary, but reflect several well-known geometric structures on exceptional symmetric spaces.


\section{A symmetry among exceptional symmetric spaces}

The inclusion diagram in Corollary \ref{main-3} exhibits a remarkable left-right symmetry. 
It is natural to ask whether this symmetry reflects some geometric relationship among exceptional symmetric spaces. 
In this section, we present several observations suggesting that this may indeed be the case.

We begin with the totally geodesic embeddings themselves.
Let $M,M',N,$ and $N'$ be symmetyric spaces in the diagram, where $M$ and $M'$ correspond to $N$ and $N'$ under the left-right symmetry, respectively.
By Corollary \ref{main-3}, there exists a totally geodesic embedding $M \to M'$ if and only if there exists a totally geodesic embedding $N \to N'$,
Thus, the existence of totally geodesic embeddings is preserved by the left-right correspondence.


The same symmetry also appears at the level of the full isometry groups.
The isometry groups of irreducible compact symmetric spaces were classified (Chapter VII, Section 4, \cite{Loos}).
For each symmetric space $M$ in the diagram, let $E(M)$ denote the number of connected components of its isometry group.
The values of $E(M)$ are shown below:
\[
\small
\xymatrix@C=12pt@R=7pt
{
 & & & E_{8}^{2} & & & \\
 & & EVIII^{1} \ar[ur] & & EIX^{1} \ar[ul] & &  \\
 & EV^{2} \ar[ur] & & EVI^{1} \ar[ul] \ar[ur] & & EVII^{2} \ar[ul] & \\
 EI^{2}  \ar[ur] & & EII^{2} \ar[ul] \ar[ur] & & EIII^{2} \ar[ul] \ar[ur] & & EIV^{2} \ar[ul] \\
 & FI^{1} \ar[ul] \ar[ur] & & & & FII^{1}, \ar[ul] \ar[ur] & \\
}
\]
where the superscripts denotes $E(M)$.
As the diagram shows,
\[
E(M) = E(M')
\]
whenever $M$ and $M'$ correspond under the left-right symmetry.
Thus, the number of connected components of the full isometry group is preserved by this correspondence.


A similar phenomenon occurs for local isometry classes.
If $M$ is not one of $EI, EIV, EV$, and $EVII$, then its local isometry class contains only $M$.
For each of the four cases, the local isometry class contains two symmetric spaces: the simply connected space $M$ and a non-simply connected quotient, called the bottom space and denoted by $M^{*}$.
The covering maps
\[
EI \to EI^{*}, \quad
EIV \to EIV^{*},
\]
have degree $3$, whereas
\[
EV \to EV^{*}, \quad
EVII \to EVII^{*}
\]
have degree $2$.
After adjoining the bottom spaces, the diagram becomes
\[
\small
\xymatrix@C=12pt@R=7pt
{
 & & & & E_{8} & & & & \\
 & EV^{*} & & EVIII \ar[ur] & & EIX \ar[ul] & & EVII^{*} & \\
EI^{*} & & EV \ar[ur] \ar[ul]_{\text{2-fold}} & & EVI \ar[ul] \ar[ur] & & EVII \ar[ul] \ar[ur]^{\text{2-fold}} & & EIV^{*} \\
 & EI  \ar[ur] \ar[ul]_{\text{3-fold}} & & EII \ar[ul] \ar[ur] & & EIII \ar[ul] \ar[ur] & & EIV \ar[ul] \ar[ur]^{\text{3-fold}} & \\
 & & FI \ar[ul] \ar[ur] & & & & FII. \ar[ul] \ar[ur] & & \\
}
\]
Thus, the number of symmetric spaces in each local isometry class is invariant under the left-right symmetry.
For the four nontrivial local isometry classes, the degree of the covering onto the bottom space is also preserved.


The symmetry becomes particularly suggestive when we consider polars.
The plolars of $F_{4}, E_{6}, E_{7}, E_{8}$ are listed in Table \ref{polars} of Subsection \ref{polars and antipodal sets}.
Their inclusion relations are as follows:
\[
\xymatrix@C=12pt@R=7pt
{
FI \ar[r]^{} \ar[dr]^{} & EII \ar[r]^{} \ar[dr]^{} & EVI \ar[r]^{} \ar[dr]^{} & EIX \ar[dr]^{} & \\
& F_{4} \ar[r]^{} & E_{6}\ar[r]^{} & E_{7}\ar[r]^{} & E_{8}. \\
FII \ar[r]_{} \ar[ur]_{} & EIII \ar[r]_{} \ar[ur]_{} & EVI \ar[r]_{} \ar[ur]_{} & EVIII \ar[ur]_{} & \\
}
\]
The polars in the first row are the polars nearest to the base point and are Wolf spaces.
Those in the third row are the polars farthest from the base point and Rosenfel planes.
These two chains correspond to the following part of the diagram:
\[
\small
\xymatrix@C=12pt@R=7pt
{
 & & & E_{8} & & & \\
 & & EVIII \ar[ur] & & EIX \ar[ul] & &  \\
 & EV \ar@{.}[ur] & & EVI \ar[ul] \ar[ur] & & EVII \ar@{.}[ul] & \\
 EI  \ar@{.}[ur] & & EII \ar@{.}[ul] \ar[ur] & & EIII \ar[ul] \ar@{.}[ur] & & EIV \ar@{.}[ul] \\
 & FI \ar@{.}[ul] \ar[ur] & & & & FII. \ar[ul] \ar@{.}[ur] & \\
}
\]
Thus, the two distinguished chains of polars are exchanged by the left-right symmetry of the inclusion diagram.


The same correspondence also appears when we distinguish totally geodesic embeddings according to reflectivity.
Let $N$ be a submanifold of $M$.
The submanifold $N$ is called a {\it reflective} submanifold, if it is a connected component of the fixed point set of an involutive isometry of $M$.
Leung classified the reflective submanifolds of compact symmetric spaces \cite{Leung}.
Using $\to$ for reflective inclusions and $\Rightarrow$ for non-reflective inclusions, the diagram becomes
\[
\small
\xymatrix@C=12pt@R=7pt
{
 & & & E_{8} & & & \\
 & & EVIII \ar@{->}[ur]_{} & & EIX \ar@{->}[ul]^{} & &  \\
 & \ EV \ar@{=>}[ur]^{\text{non-reflective \quad\ }} & & EVI \ar@{->}[ul]^{} \ar@{->}[ur]_{} & & EVII \ar@{=>}[ul]_{\text{\quad\ \  non-reflective}} & \\
 \ EI \ \ar@{=>}[ur]^{\text{non-reflective \quad\ }} & & EII \ar@{->}[ul]^{} \ar@{->}[ur]_{} & & EIII \ar@{->}[ul]^{} \ar@{->}[ur]_{} & & EIV \ar@{=>}[ul]_{\text{\quad\ \  non-reflective}} \\
 & FI \ar@{->}[ul]^{} \ar@{->}[ur]_{} & & & & FII. \ar@{->}[ul]^{} \ar@{->}[ur]_{} & \\
}
\]
The four inclusions
\[
EI \subset EV, \qquad EV \subset EVIII, \qquad EIV \subset EVII, \qquad EVII \subset EIX
\]
are not reflective.
Nevertheless, the embeddings
\[
\begin{array}{llc} \vspace{1mm}
(S^{1} \times EI)/\mathbb{Z}_{3} \subset EV, & (S^{1} \times EIV)/\mathbb{Z}_{3} \subset EVII, \\
(S^{2} \times EV)/\mathbb{Z}_{2} \subset EVIII, & (S^{2} \times EVII)/\mathbb{Z}_{2} \subset EIX \\
\end{array}
\]
are reflective embeddings.
Thus, reflectivity is compatible with the left-right correspondence.
More strikingly, each of the four non-reflective inclusions becomes reflective after adjoining an additional irreducible factor, and the required factor is itself preserved by the correspondence: $S^{1}$ for $EI \subset EV$ and $EIV \subset EVII$, and $S^{2}$ for $EV \subset EVIII$ and $EVII \subset EIX$.


The left-right correspondence extends further to the structure of maximal totally geodesic submanifolds.
Recently, Kollross and Rodr\'{i}guez-V\'{a}zquez classified the maximal totally geodesic submanifolds of exceptional symmetric spaces of noncompact type \cite{Kollross}.
For each symmetric space $M$ in the diagram, let $L_{M}$ denote the set of isometry classes of maximal totally geodesic submanifolds of $M$ that  contain one of the exceptional symmetric spaces in the diagram as irreducible factor.
In Table \ref{maximal totally}, we list $L_{M}$ for each case.
As the table shows, if $M$ and $M'$ correspond under the left-right symmetry, then the elements of $L_{M_{1}}$ and $L_{M_{2}}$ are paired by the same correspondence.
Thus, the symmetry persists even at the level of maximal totally geodesic submanifolds containing exceptional symmetric spaces as irreducible factors.

\begin{table}[htbp]
\caption{Maximal totally geodesic submanifolds} \label{maximal totally}
\centering
\begin{tabular}{c|cccccccccc} \hline
$M$ & & & \hspace{15mm} $L_{M}$  \\ \hline\hline
$EI$ & $FI$ \\  
$EII$ & $FI$ \\
$EIII$ & $FII$ \\
$EIV$ & $FII$ \\ \hline
$EV$ & $EII$ & $(S^{1} \times EI)/\mathbb{Z}_{3}$ & $S^{2} \times FI$ \\
$EVI$ & $EII$ & $EIII$ \\
$EVII$ & $EIII$ & $(S^{1} \times EIV)/\mathbb{Z}_{3}$ & $S^{2} \times FII$ \\ \hline
$EVIII$ & $EVI$ & $(S^{2} \times EV)/\mathbb{Z}_{2}$ & $(SU(3)/SO(3) \times EI)/\mathbb{Z}_{3}$ & $\mathbb{C}P^{2} \times EII$ & $G_{2}/SO(4) \times FI$ \\
$EIX$ & $EVI$ & $(S^{2} \times EVII)/\mathbb{Z}_{2}$ & $(SU(3)/SO(3) \times EIV)/\mathbb{Z}_{3}$ & $\mathbb{C}P^{2} \times EIII$ & $G_{2}/SO(4) \times FII$ \\ \hline
\end{tabular}
\end{table}

As the table shows, if $M$ and $M'$ correspond under the left-right symmetry, then the elements of $L_{M_{1}}$ and $L_{M_{2}}$ are paired by the same correspondence.
Thus, the symmetry persists even at the level of maximal totally geodesic submanifolds containing exceptional symmetric spaces as irreducible factors.


Finally, the symmetry is already encoded in the conjugation-orbit structure underlying our construction. 
Recall the two diagrams in Theorem \ref{main-2}:
\[
\footnotesize
\hspace{-2mm}
\xymatrix@C=5pt@R=11pt
{
 & & FII \ar[d] & & FII \ar[d] & & \\
 & & EIII \ar[d] & & EIII \ar[d] & & \\
FI \ar[r] & EII \ar[r] & EVI \ar[dr] & & EVII \ar[dl] & EIII \ar[l] & FII \ar[l] \\
 & & & EIX & & & \\
FII \ar[r] & EIV \ar[r] & EVII \ar[ur] & & EVII \ar[ul] & EIV \ar[l] & FII \ar[l] \\
 & & EIV \ar[u] & & EIV \ar[u] & & \\
 & & FII \ar[u] & & FII, \ar[u] & & \\
}
\ 
\xymatrix@C=5pt@R=11pt
{
 & & FI \ar[d] & & FI \ar[d] & & \\
 & & EII \ar[d] & & EII \ar[d] & & \\
FII \ar[r] & EIII \ar[r] & EVI \ar[dr] & & EV \ar[dl] & EII \ar[l] & FI \ar[l] \\
 & & & EVIII & & & \\
FI \ar[r] & EI \ar[r] & EV \ar[ur] & & EV \ar[ul] & EI \ar[l] & FI \ar[l] \\
 & & EI \ar[u] & & EI \ar[u] & & \\
 & & FI \ar[u] & & FI. \ar[u] & & \\
}
\]
In these two diagrams, the orbit types and the multiplicities correspond under the left-right symmetry.
For instance, the left diagram contains three orbits of type $EIII$, whereas the right diagram contains three orbits of type $EII$.
The inclusion relations among the individual orbits are also preserved.
For example, the left diagram contains two chains of the form
\[
FII \to EIII \to EVII \to EIX,
\]
while the right diagram contains two corresponding chains of the form
\[
FI \to EII \to EV \to EVIII.
\]

The recurrence of the same correspondence across these different geometric settings points to a systematic and unexpected relationship among exceptional symmetric spaces. Whether this relationship can be understood in terms of a more fundamental geometric principle, or even a kind of duality, remains an open question.

\end{document}